\documentclass[a4paper,12pt]{amsart}
\usepackage[utf8]{inputenc}
\usepackage{lmodern}
\usepackage{amsthm}
\usepackage{amssymb}
\usepackage{amsmath,amscd}
\usepackage{enumitem,mathrsfs}

\newtheorem{definition}{Definition}
\newtheorem{conjecture}{Conjecture}
\newtheorem{lem}{Lemma}[section]
\newtheorem{thm}{Theorem}

\newtheorem{corollary}[lem]{Corollary}
\newtheorem{prop}[lem]{Proposition}
\newtheorem*{remark}{Remark}
\newcommand{\C}{\mathbb{C}}
\newcommand{\Q}{\mathbb{Q}}

\newcommand{\Z}{\mathbb{Z}}
\newcommand{\G}{\mathbb{G}}

\newcommand{\R}{\mathbb{R}}

\newcommand{\fn}{f^{\circ n}}
\newcommand{\ord}{\text{ord}}
\author{Harry Schmidt}
\address{Harry Schmidt \\ Universität Basel \\ Departement Mathematik und Informatik \\ Spiegelgasse 1, 4051 Basel, harry.Schmidt@unibas.ch}
\title[Polynomial dynamics and local analysis]{Polynomial dynamics and local analysis of small and grand orbits}
\subjclass{Primary: 37F10 Secondary: 14G20}
\begin{document}

\begin{abstract} We prove an analog of Lang's conjecture on divisible groups  for polynomial dynamical systems over number fields. In our setting the role of the divisible group is taken by the small orbit of a point $\alpha$ where  the small orbit by a polynomial $f$  is given by
\begin{align*}
\mathcal{S}_\alpha = \{\beta \in \C; f^{\circ n}(\beta) = f^{\circ n}(\alpha) \text{ for some } n \in \Z_{\geq 0}\}.
\end{align*}
Our main theorem is a  classification of the algebraic relations that hold between infinitely many pairs of points in $\mathcal{S}_\alpha$ when everything is defined over the algebraic numbers and the degree $d$ of $f$ is at least 2.
Our proof relies on a careful study of localizations of the dynamical system and follows an entirely different approach than previous proofs in this area. In particular we introduce transcendence theory and Mahler functions into this field. 
Our methods also allow us to  classify all algebraic relations that hold for infinitely many pairs of points in the grand orbit 
\begin{align*}
\mathcal{G}_\alpha = \{\beta \in \C; f^{\circ n}(\beta) = f^{\circ m}(\alpha) \text{ for some } n ,m\in \Z_{\geq 0}\}
\end{align*}
 of $\alpha$ if $|f^{\circ n}(\alpha)|_v \rightarrow \infty$ at a finite place $v$ of good reduction co-prime to $d$ . 
 
\end{abstract}
\maketitle 
\section{Introduction}
Let $K$ be a number field and $\alpha \in K$. We study the dynamical system attached to a polynomial $f \in K[X]$ of degree $d \geq 2$ and denote by $\fn$ the $n$-th iterate of $f$ defined by $f^{\circ 0}(X) = X$ and $f^{\circ n}(X) = f^{\circ n-1}(f(X))$ for $n \in \Z_{\geq 1}$. We consider the small orbit $\mathcal{S}_\alpha$ of a point $\alpha$ that was introduced by McMullen and Sullivan \cite[6.1]{mcmullen} as part of their investigations of the dynamics of homolorphic maps. It is given by 
\begin{align*}
\mathcal{S}_\alpha= \{\beta \in \overline{\Q}; f^{\circ n}(\beta) = \fn(\alpha) \text{ for some } n \in \Z_{\geq 0}\}.
\end{align*}
This $\mathcal{S}_\alpha$ is an infinite set unless $\alpha$ is a super-attractive fixed point of degree $d$ (see \cite{milnor} for a definition).  In order to have a more traditional diophantine analogue we may consider $\mathcal{S}_\alpha$ to be the set of  ``torsion translates" of $\alpha$ for the dynamical system attached to $f$.  For example in the more classical setting of the multiplicative group $\G_m$ an element $x$ of $\alpha \mu_{\infty}$ (where $\mu_{\infty}$ is the group of roots of unity) satisfies $x^n = \alpha^n$ for some positive integer $n$. 
The study of the equation $x^n = y^n$ where $y$ varies in  a finitely generated group is integral to go from the Mordellic part to Lang's full conjecture in $\G_m^n$ \cite{laurent}. In fact  employing Kummer theory one reduces the conjecture to the study of this equation (and thus Manin-Mumford with a certain uniformity). The same approach is used in all known proofs of  Mordell-Lang to go from the Mordellic part to the divisible hull (for example   in the ultimate resolution of the conjecture for semi-abelian varieties by McQuillan \cite{mcquillan}). \\
In this article we obtain a structure theorem for the set $\mathcal{S}_\alpha$ that is analogous to what Manin-Mumford delivers in the context of algebraic groups (and also achieve a certain uniformity). We also use a similar strategy in this context to show, quite analogously as in the group context,  that we can go from the small orbit of $\alpha$ to its grand orbit
\begin{align*}
\mathcal{G}_\alpha = \{\beta \in \overline{\Q}; f^{\circ n}(\beta) = f^{\circ m}(\alpha) \text{ for some } n ,m\in \Z_{\geq 0}\},
\end{align*}
albeit under the condition that the orbit of $\alpha$ tends to infinity at a finite place of good reduction co-prime to $d$. These results may be viewed as new cases of the Zilber-Pink (or unlikely intersection) conjecture on dynamical systems as formulated in \cite[p.1438]{ZPdynsystems} as we will explain below. In order to ease notation and deal with trivial cases and ones covered by previous results we introduce two definitions (one of which appeared previously in the literature). 
\begin{definition} Let $\mathcal{K}$ be a field and $\overline{\mathcal{K}}$ its algebraic closure. 
\begin{enumerate} 
\item[(i)]We say that a curve $\mathcal{C} \subset \mathbb{A}^2$ defined over a field $\mathcal{K}$ is fibral if it is of the form 
	$\mathcal{C} = \{\gamma\}\times \mathbb{A}^1$ or $\mathcal{C} = \mathbb{A}^1\times \{\gamma\}$ for $\gamma \in \mathcal{K}$.
	\item[(ii)] We call a polynomial $f$ exceptional if there exists a linear polynomial $L \in \overline{\mathcal{K}}[X]$ such that $L\circ f\circ L^{-1}$ is equal to $X^d$ or to a Chebychev polynomial or a negative Chebychev polynomial. Otherwise we call a polynomial non-exceptional. 
\end{enumerate} 
\end{definition} 
Our definition of exceptional agrees with that in \cite{dynbogomolovcurves} for polynomials. For the notion of pre-periodic varieties we refer the reader to the established literature \cite{medvedevscanlon,dynbogomolovcurves}.  Our main theorem is the following. 

\begin{thm} \label{thm} Let $K $ be a number field and $\mathcal{C} \subset \mathbb{A}^2$ be an irreducible plane curve that is not fibral. Let $f \in K[X]$ of degree $d \geq 2$ be a non-exceptional polynomial and $\alpha \in K$ be a point that is not pre-periodic under $f$. If 
\begin{align*}
\mathcal{C}(\overline{\Q})\cap \mathcal{S}_\alpha^2
\end{align*}
contains infinitely many points, then there exists $n \in \Z_{\geq 0}$ such that 
$$(f^{\circ n},f^{\circ n})(\mathcal{C}) = \Delta,$$ 
where $\Delta \subset \mathbb{A}^2$ is the diagonal. 
\end{thm}
If $\mathcal{C}$ is fibral then the statement is trivial and $\mathcal{C}$ is, up to permuting coordinates,  of the form $\{\gamma\}\times\mathbb{A}^1$ for $\gamma \in \mathcal{S}_\alpha$. 
If $f$ is an exceptional polynomial then  Manin-Mumford for the multiplicative group (see section \ref{reduction}) answers most questions one might have about Zariski-closures of subsets of $\mathcal{S}_\alpha^2$. If $\alpha$ is pre-periodic for $f$ then $\mathcal{S}_\alpha$ is contained in the set of pre-periodic points of $f$ and in this case the dynamical Bogomolov conjecture for split rational maps proven by Ghioca, Nguyen and Ye \cite[Theorem 1.1]{dynbogomolovcurves} applies. We will combine these theorems in a unified statement in Theorem \ref{thm1}. We remark that $\Delta$ is fixed by $(f,f)$ but not all pre-periodic curves in $\mathbb{A}^2$ are pull-backs of the diagonal (see for example \cite{medvedevscanlon}).  \\

Moreover an extension of Theorem \ref{thm} to either $f$ being exceptional or $\alpha$ being pre-periodic fails. A whimsical explanation being that we have more structure lurking in the background if either of these conditions hold. More concretely, if $f = X^d$ and thus exceptional we can take any (co-prime) integers $l,k$ and consider the co-set in $\G_m^2$ defined by
\begin{align}\label{coset}
X^{\ell}Y^k =\alpha^{\ell+k}. 
\end{align}
Then any of these co-sets contain infinitely many points in  $\mathcal{S}_\alpha^2$ but if we take $\ell,k$ to be co-prime to $d$ say it does not define a curve as in the conlusion of the Theorem. If $f$ is non-exceptional then the only superattracting fixed point of degree $d$ is at infinity. Thus if $f$ is non-exceptional and  $\alpha$ is pre-periodic, say $f^{\circ \ell}(\alpha) = f^{\circ k}(\alpha)$ with $\ell > k$, then the rational curve given by $(f^{\circ \ell}(X), f^{\circ k}(X))$ provides a counter-example to a possible extension of Theorem \ref{thm}. Thus Theorem \ref{thm} is best-possible.\\

Another way of putting Theorem \ref{thm} is to say that  the only non-trivial algebraic relations that can hold between infinitely many pairs of points in $\mathcal{S}_\alpha$ are the ones that define $\mathcal{S}_\alpha$ in the first place. It is straightforward to check that the reverse holds true as well. Each curve that appears in the conclusion of Theorem \ref{thm} also contains infinitely many points on $\mathcal{S}_\alpha^2$. \\

It is not foreseeable that what is  obtained in this article can be achieved with previously established methods such as applying known equidistribution 
results or Siegel's theorem on integral points. This is because the usual basic assumptions such as small canonical height respectively rationality of the points under consideration are not fulfilled in the cases we consider.  \\
Our methods are close to  transcendence techniques which introduce entirely new techniques into the study of dynamical systems. However they also lead to a new and elementary proof of Manin-Mumford for $\G_m^n$ as was demonstrated in \cite{mm}. We expect there to be  more applications of the methods used in this article. \\

In the assumption of Theorem \ref{thm},  $\alpha$ is not preperiodic and thus has positive canonical height \cite{dynbogomolovcurves}. This implies that  there  exists a place $v$ of $K$ for which $|f^{\circ n}(\alpha)|_v \rightarrow \infty$ (where we denote by $|\cdot|_v$ the corresponding $v$-adic absolute value) and consequently for which the local canonical height of $\alpha$ is positive. All points in $\mathcal{S}_\alpha$ have the same local canonical height at all places of $K$ and if there is an archimedean place at which the local height is positive then Theorem \ref{thm} follows from Theorem \ref{equipotential} below; if there is such a non-archimedean place then it follows from Proposition \ref{prop1} in the next section. \\
In fact, in each case, we are able to show a stronger version that allows for some uniformity (see Proposition \ref{prop1} and Theorem \ref{uniform}). If $v$ can be chosen to be finite and of good reduction co-prime to $d$ we can go from small orbits to grand orbits.  Quite similarly as for the case of finitely generated to finite rank groups in Lang's conjecture on semi-abelian varieties that we  alluded to above.
\\
We say that $v$ is a finite place of good reduction (for $f$) if $f = a_0X^d + \sum_{i = 1}^{d}a_iX^{d-i}$ and $|a_i/a_0|_v \leq 1$ for $i = 1, \dots, n$. We say that $v$ is co-prime to $d$ if $|d|_v = 1$. 

\begin{thm} \label{grandorbit} Let $K,\mathcal{C}, \Delta$ and $f$ be as in Theorem \ref{thm}. Further suppose that $|f^{\circ n}(\alpha)|_v \rightarrow \infty$ at a finite place $v$ of $K$ of good reduction that is co-prime to $d$. If 
\begin{align*}
\mathcal{C}(\overline{\Q})\cap \mathcal{G}_\alpha^2
\end{align*}
contains infinitely many points, then there  exist $n,m \in \Z_{\geq 0}$ such that 
$$(f^{\circ n},f^{\circ m})(\mathcal{C}) = \Delta.$$
\end{thm} 
We may view the above as resolving the Zilber-Pink conjecture as formulated in \cite{ZPdynsystems} for the curve $\mathcal{C} \times \{\alpha\} \subset (\mathbb{P}_1)^3$ and the dynamical system $(f,f,f)$. \\ 
As an example of an application of Theorem \ref{grandorbit} we can take the non-exceptional polynomial $f = X^2 -1$ and $\alpha = 1/3$. Then $\mathcal{G}_\alpha$ is dense in the filled Julia set of $f$ (in $\C$) and Theorem  \ref{grandorbit} shows that the only algebraic relations that can hold between infinitely many distinct pairs of elements in the grand orbit of $1/3$ are the same that define the grand orbit. We could also take the same $f$ and $\alpha = 5/3$ and then $\mathcal{G}_\alpha$ is dense in the complement of the Julia set and our theorem still holds. This gives a hint at the algebraic rigidity of the extended tree formed by the grand-orbit under $f$ at a non-pre-periodic point; at least if $f$ is non-exceptional. In contrast if $f = X^d$ then we can take the same co-set as above (\ref{coset}) with $k,\ell$ co-prime to $d$ to obtain a counterexample to an extension of Theorem \ref{grandorbit} to exceptional polynomials. 
Again one could phrase the above theorem as stating that the only non-trivial algebraic relations that hold between infinitely many pairs of points in $\mathcal{G}_\alpha$ are the ones that define the set itself. Also in section \ref{remarks} we show that Theorem \ref{grandorbit} can not be  extended to $\alpha$ being a pre-periodic point. For the proof of Theorem \ref{grandorbit} we also rely on work of Ghioca, Nguyen and Ye on the dynamical Bogomolov conjecture \cite[Theorem 1.1]{dynbogomolovcurves}. We want to point out that we really use the Bogomolov property, not just the Manin-Mumford property in their work. 

\begin{remark} Using the classification theorem of Medvedev and Scanlon once can see that all pre-periodic curves that have non-empty intersection with $\mathcal{S}_\alpha^2$ or $\mathcal{G}_\alpha^2$ are as in the conclusion of Theorem \ref{thm} or \ref{grandorbit}. For a proof see Lemma \ref{reductionprep} in section \ref{reduction}.
\end{remark}

If there is an infinite place $v$ of $K$ at which $f^{\circ n}(\alpha)$ tends to infinity we deduce Theorem \ref{thm} from a new rigidity theorem (Theorem \ref{equipotential} below) that holds for non-exeptional polynomials. But we need some notation first. \\

For any complex polynomial $f \in \C[X]$ of degree $d \geq 2$ we let 
\begin{align*}
g_f(z) = \lim_{n \rightarrow \infty} \frac{\log^+(|f^{\circ n}(z)|)}{d^n}
\end{align*}
where $\log^+(x) = \log \max \{1, x\}$ for $x \in \R_{\geq 0}$. This $g_f$ is the Green function attached to the Julia set of $f$ and a continuous function on $\C$. For $r > 0$ we define the equi-potential curve of $r$ to be the level-set
\begin{align}\label{Lr} 
L_r = \{z \in \C; g_f(z) = r \}. 
\end{align}
We note that for $f = X^d$ it holds that $L_r$ is the circle of radius $\exp(r)$ and that for all $f$ the level-set at 0 is the filled Julia set of $f$. Thus $\alpha \in \C$ satisfies  $|f^{\circ n}(\alpha)| \rightarrow \infty$ with the usual absolute value if and only if $g_f(\alpha) >0$. \\

The following is a characterization of curves that intersect $L_r^2$ in infinitely many points.  
\begin{thm}\label{equipotential} Let $f$ be a non-exceptional polynomial with complex coefficients. If a complex connected plane curve $\mathcal{C} \subset  \mathbb{A}^2$, that is not fibral, has infinite intersection with $L_r^2$ (for some $r > 0$), then $\mathcal{C}$ is a component of a curve defined by 
\begin{align*}
f^{\circ n}(X) - L\circ f^{\circ n}(Y) = 0.
\end{align*}
for some $n \in \Z_{\geq 0}$ and some linear polynomial $L(X)$ such that $L\circ f^{\circ n} = f^{\circ n} \circ L$.  Here $X,Y$ are the affine coordinates on $\mathbb{A}^2$. 
\end{thm}
Due to the fact that $L_r$ is a piecewise real analytic set one can interpret Theorem \ref{equipotential} as a characterization of all algebraic correspondences that preserve an equi-potential curve $L_r$. These sort of rigidity theorems often play a key role in the proofs in arithmetic dynamics. For example in the proof of a case of the dynamical André-Oort conjecture in \cite{dynandreoort}, \cite{favregauthier} a key ingredient is a classification of algebraic correspondences preserving (generalized) Mandelbrot set.  \\


Theorem \ref{equipotential} does not hold without the requirement on $f$ to be non-exceptional. For example taking $f = X^d$  graphs of  Möbius transformations sending $L_r$ to $L_r$ provide an ample supply of counterexamples (see for example the introduction of \cite{bayshabegger}). These counterexamples are made possible by the fact that  $L_r$ is semi-algebraic. This suggests that $L_r$ should be a transcendental curve if $f$ is not exceptional and we can support this intuition with the following theorem whose proof is  straightforward given the tools used in the proof of Theorem \ref{equipotential}. 
\begin{thm}\label{transc}
$L_r$ contains a connected semi-algebraic curve for some $r > 0$ if and only if $f$ is exceptional.
\end{thm}
We pause to mention that we thought it is worth recording the above theorem because it is analogous to Fatou's theorem \cite[p.250]{fatou} since Fatou's theorem characterises exceptional polynomials as those, whose Julia set is a (piecewise) analytic arc while our theorem characterizes exceptional polynomials through properties of $L_r$. As the curves $L_r$ are always piecewise real analytic we need a finer criterion and Theorem \ref{equipotential} shows that this finer criterion is (real) algebraicity.  
It is also worth pointing out that Fatou's theorem itself is crucial for the proof of Theorem \ref{equipotential} and Theorem \ref{transc}. \\
Another important characterisation of exceptional polynomials that we will make use of is due to the combined work of Levin \cite{levinsymmetries} and Douady-Hubbard-Thurston \cite{DH}. This characterises polynomials whose Julia set has a normal family of symmetries as being the exceptional ones. These theorems allow us to show that ``irrational rotations" of the basin of infinity of $f$ with connected Julia set can not be algebraic (see Lemma \ref{rotation}).\\
Finally, for $f$ with disconnected Julia set we introduce considerations that seem to be entirely absent from the existing literature. Namely we study the monodromy action on Mahler functions that appear as the inverse of the famous Böttcher coordinates. This requires careful analysis of Puiseux expansions among other things. Indeed, somewhat surprisingly in the disconnected case we do not need to use any deep rigidity theorem from complex dynamics. Instead mondromy alone seems to contain sufficient information in this case. 
\\ 

As a final remark on the rigidity theorems we note that the techniques used to prove Theorem \ref{equipotential} and \ref{transc} potentially extend to higher dimensions to prove that for a non-exceptional polynomials the components of the intersection of an algebraic variety with powers of $L_r$ must have low dimension. This might be of interest in connection with dynamical questions and equidistribution on powers of $L_r$. \\

At this point it seems appropriate to put Theorem \ref{thm} and \ref{grandorbit} into the context of the dynamical literature as well.
Our paper (and in particular Theorem \ref{grandorbit}) may be viewed as a first attempt to unify the dynamical Mordell-Lang conjecture and the dynamical Manin-Mumford conjecture. \\
A dynamical Mordell-Lang conjecture was first formulated by Ghioca and Tucker in \cite{GTJNT} and then further expanded on in the book \cite{bookdml} where the authors also discuss intersections of orbits. In \cite{GTJNT} the authors also use properties of attracting fixed points in their proof. A dynamical Manin-Mumford conjecture was first formulated by Zhang and later corrected and refined by Ghioca, Tucker and Zhang \cite{towardsdmm} as well as by Ghioca and Tucker \cite{dmmconj}. \\
  
 On the dynamical Manin-Mumford side we mention some results that mostly concern the projective line $\mathbb{P}_1$ and its powers $\mathbb{P}_1^n$. As alluded to above, Ghioca, Nguyen and Ye have  proven the dynamical Manin-Mumford conjecture for split endomorphisms of $\mathbb{P}_1^n$. This work relied on the equi-distribution theorem by Yuan on adelic heights \textit{op. cit.}. Similarly Baker and Demarco used equidistribution techniques to prove a Manin-Mumford theorem for families of dynamical systems \cite{bakerdemarco}. This is by no means an exhaustive list of the results obtained by using the powerful equi-distribution theorem combined with deep insights into the intricate geometry and arithmetic of dynamical systems. Worth mentioning is also the resolution of some cases of a dynamical André-Oort conjecture with these techniques \cite{dynandreoort}. \\
 On the dynamical Mordell-Lang side the works that are perhaps closest in spirit to ours are \cite{dtz}, \cite{dtzduke} as they treat intersections of varieties with orbits of polynomials. \\
 From the point of view of height theory Theorem \ref{thm} may also be viewed as an intermediate between the work \cite{dynbogomolovcurves} and \cite{dtz} as we are treating points of constant canonical height whereas the works just mentioned treat points of small respectively unbounded canonical height. \\

In order to incorporate exceptional polynomials and periodic points in our statement we need a bit of care to encode essentially group theoretic information in dynamical terms (see \cite{dmmconj} for related problems). What is new in our context compared to \cite{dmmconj} is that we need to define an analogue of co-sets in the dynamical setting.  For this we introduce some notation. We formulate a conjecture for powers of the projective line here and in section \ref{remarks} a conjecture for all projective varieties. 

For an endomorphism $g $ of $\mathbb{P}_1$ we let $\Xi_g$ be $\mathbb{P}_1\setminus\text{S}_g$ where $\text{S}_g$ is the set of superattracting fixed points of $g$ of degree $\deg(g)$ \cite{milnor}. The variety $\Xi_g$ is always  quasi-affine. It is a natural space on which $g$ acts. We say that two morphisms commute if they commute by composition.  
\begin{definition} \label{def1} Let $\mathcal{K}$ be an algebraically closed field of characteristic 0 and $n$ a positive integer. We say that a  variety $\mathcal{V} \subset \mathbb{P}_1^n$ defined over $\mathcal{K}$ is stratified by an endomorphism $g$ of $\mathbb{P}_1$ defined over $\mathcal{K}$ if there exists an endomorphism $h$ of $\mathbb{P}_1$ and a finite morphism $\varphi:\Xi_h \rightarrow \Xi_g$ with the property $g\circ \varphi = \varphi\circ h$ such that the following holds:  There exists a variety $\mathcal{V}_\varphi \subset \Xi_h^n$ and a morphism $H: \Xi_h^n \rightarrow \Xi_h^n$ commuting with $(h,\dots, h):\Xi_h^n \rightarrow \Xi_h^n$ such that  
\begin{align*}
(\varphi, \dots, \varphi)(\mathcal{V}_\varphi) = \mathcal{V}\cap \Xi_g^n, \quad (\varphi, \dots, \varphi) \circ H(\mathcal{V}_\varphi) = \{\beta \} \times \Xi_g^{\dim(\mathcal{V})}
\end{align*}
for some $\beta \in \Xi_g^{n-\dim(\mathcal{V})}(\mathcal{K})$. 
\end{definition}
Definition \ref{def1} is quite technical. However it tries to capture the fact that if a rational map comes from an algebraic group, the group structure gives rise to ``many" endomorphisms that commute with the given rational map in purely dynamical terms. A simple example is that $\mathcal{V}$ is stratified by $g$ if the projection of $\mathcal{V} $ to a factor of $\mathbb{P}_1$ is constant.  In the exceptional cases we can think of $\Xi_h$ as  coming from an algebraic group and $H$ as coming from a group morphism.   
Note that for $g= X^2$ any co-set of $\G_m^n$ defines a variety that is stratified by $g$.  We obtain the following from Theorem \ref{thm}, \cite{dynbogomolovcurves} and Manin-Mumford.   
\begin{thm}\label{thm1} Let $f \in K[X]$ be a polynomial of degree $d \geq 2$ and let $\alpha \in K$. Let $\mathcal{C} \subset \mathbb{A}^2$ be a geometrically irreducible curve over $K$. If  $\mathcal{C} \cap \mathcal{S}_\alpha^2$ contains infinitely many points then $\mathcal{C}$ is  either pre-periodic under the endomorphism $(f,f)$ or is stratified by $f$.
\end{thm}

Following a suggestion of Ghioca we formulate a conjecture for intersections of varieties with grand orbits. It can be viewed as a natural dynamical analogue of Lang's conjecture on finite rank groups.\\

For an endomorphism $g$ of $\mathbb{P}_1$ and $\alpha \in \mathbb{P}_1(\mathcal{K})$ we set $\mathcal{G}_\alpha(g)$ to be the grand-orbit of $\alpha$ under $g$.  
Let $X = \mathbb{P}_1^n$ and  $\underline{g} = (g_1, \dots, g_n)$ where $g_1,\dots, g_n$ are endomorphisms of $\mathbb{P}_1$ of degree at least 2 defined over $\mathcal{K}$ as above.  Let $\underline{\alpha} = (\alpha_1,\dots, \alpha_n)\in \mathbb{P}_1^n(\mathcal{K})$ and let $V$ be a hypersurface in $X$. 
\begin{conjecture}\label{conj} Suppose that
\begin{align*}
V \cap \mathcal{G}_{\alpha_1}(g_1)\times \dots \times \mathcal{G}_{\alpha_n}(g_n)  
\end{align*}
contains infinitely many points. Then $V$ is either pre-periodic under $\underline{h} = (h_1,\dots, h_n)$ where $h_1, \dots, h_n$ are endomorphisms commuting with an iterate of $g_1, \dots, g_n$ respectively and such that at least one of $h_1, \dots, h_n$ has degree at least 2 or $V$ is stratified by $g_i$ for some $i \in \{1,\dots, n\}$. 
\end{conjecture}
From Theorem \ref{grandorbit} follows that we can treat part of Conjecture \ref{conj}. 

\begin{thm}\label{grandorbits} Let $f \in K[X]$  and let $\alpha \in K$ be such that $|f^{\circ n}(\alpha)|_v \rightarrow \infty$ at a finite place $v$ of good reduction co-prime to $d$. Then Conjecture \ref{conj} holds true for $\underline{\alpha} = (\alpha, \alpha)$ and $\underline{f} = (f,f) $. 
\end{thm}

We call a curve $\mathcal{C} \subset \mathbb{A}^2$ special if it is a component of curve defined by $f^{\circ n}(X)- f^{\circ n}(Y) = 0$ for some $n \geq 0 $ or if $\mathcal{C}$ is of the form $\{\beta\} \times \mathbb{A}^1$ or $\mathbb{A}^1\times \{\beta\} $ for some $\beta \in \mathcal{S}_\alpha$. 

It is also possible to put Theorem 1 in the more classical context of unlikely intersections. More precisely we define the set of special subvarieties $\text{Sp}_f$ of $\mathbb{A}^3$ with coordinates $(X_1,X_2,X_3)$ to be Boolean combinations of varieties defined by
\begin{align*}
f^{\circ n}(X_i) = f^{\circ n}(X_j), n \in \Z_{\geq 0}, i,j \in \{1,2,3 \}, i \neq j.
\end{align*} 
Theorem \ref{thm} implies that for $\alpha$ that is not pre-periodic under $f$ and $f$ non-exceptional, the intersection 
 \begin{align*}
 \mathcal{C} \times \{\alpha\} \cap (\cup_{V \in \text{Sp}_f, \text{codim}(V) > 1} V)    
\end{align*}
is finite unless $\mathcal{C}\times \{\alpha\}$ is contained in a special subvariety of positive codimension. \\

It is useful to keep this terminology and we will often refer to a special  curve in $\mathbb{A}^2$ to be either of the form $\{\beta\}\times \mathbb{A}^1$, respectively  $\mathbb{A}^1\times \{\beta\}$ for some $\beta \in \mathcal{S}_\alpha$ or if $\mathcal{C}$ is a component of the variety  defined by $f^{\circ n}(X) = f^{\circ n}(Y)$. \\ 

We now summarize the discussion of our methods and compare them to the Pila-Zannier method which is often used in special circumstances. As mentioned above we localize at a place $v$ for which $\alpha$ lies in its associated basin of infinity. If the place turns out to be archimedean we reduce Theorem \ref{thm} to a rigidity theorem (Theorem \ref{equipotential}) that is really an analytic property of the dynamical system. If $v$ is non-archimedean however we make use of the ``uniformization" of the dynamical system by the inverse of the Böttcher map. Quite similarly as in the Pila-Zannier method we study the pullback of the curve $\mathcal{C}$ by this uniformization and show a functional transcendence result for this pull-back. However we only show that the pull-back is not contained in a co-set of $\G_m^2$ unless $\mathcal{C}$ is special (Theorem \ref{transcendence2} and \ref{transcendence1}). We essentially study pairs of roots of unity on our pull-back and these are of bounded height but not of bounded degree. This rules out the possibility to directly apply counting of rational points in the Pila-Wilkie style. \\
Our approach is different and seems to be particularly suited for non-archimedean completions. To a fixed solution  $(\beta_1,\beta_2) \in \mathcal{S}_\alpha^2\cap \mathcal{C}$ we construct an auxiliary function that vanishes at the pull-back of the whole Galois-orbit of $(\beta_1,\beta_2)$.  We then apply Poisson-Jensen to bound its zero-set, which does not care about the degree of the point. Our method is suited for generalisations and we will explore it further in the function field case in future work. As we showed in \cite{mm}, and have alread mentioned, it  leads to a new and quite elementary proof of (the classical) Manin-Mumford in the multiplicative group. \\

As we will often work with roots of unity we set $\mu_{\infty}$ to be the set of all roots of 1 and $\mu_{d^{\infty}}$ the set of roots of unity whose order divides a power of $d$. Finally for a positive  integer $N$ we set $\mu_N$ to be the set of $N$-th root of 1. For a field $\mathcal{K}$ (often $K$) we set $\mathcal{K}_v$ to be its completion by the valuation $v$ and set $\overline{\mathcal{K}}$ to be its algebraic closure. We denote by $h,H:\overline{\Q} \rightarrow \R$ the logarithmic respectively multiplicative (naive) Weil-height \cite{bombierigubler}. \\

Here is how we proceed. In the next section we reduce the last two Theorems in this section to Theorem \ref{thm} and Theorem \ref{grandorbit}. Then we introduce the main players in our proof which are the Mahler functions defined by $f$. In section \ref{nonarchplaces} we give the proof of a uniform version of Theorem \ref{thm} when $|f^{\circ n}(\alpha)|_v \rightarrow \infty$ at a finite place. In section \ref{connectedjulia} we give the proof of Theorem \ref{equipotential} when the Julia set of $f$ is connected and  in section \ref{disconnectedjulia} when it is disconnected. We summarize the proof of Theorem \ref{thm} in section \ref{proofofthm1} and in section \ref{proofofthm2} we give the proof of Theorem \ref{grandorbit}. We conclude in the final section with some remarks and give a uniform version of Theorem \ref{equipotential}.

\section{Reductions} \label{reduction}
We first note that if Theorem \ref{thm} holds for a polynomial $f$ then it holds for $L\circ f\circ L^{-1} $ for any linear polynomial $L$. The same holds for Theorem \ref{equipotential}. Thus we may and will assume that $f$ is monic.  \\
We now reduce  Theorem \ref{thm1} to Theorem \ref{thm}.  
If $\alpha$ is pre-periodic under $f$ then Theorem \ref{thm1} follows from Theorem 1.1 in \cite{dynbogomolovcurves} and this case is taken care of.  Now assume that $f$ is exceptional. Then we may assume that $f$ is either Chebychev or a power map. In both cases there exists a morphism $\varphi: \G_m \rightarrow \Xi_f$ such that $f\circ \varphi(X) = \varphi(X^d)$. In the former cases $\varphi(X) = X + 1/X$ and in the latter it is $\varphi(X) = X$. It follows that for  each $(\beta_1, \beta_2) \in \mathcal{C} \cap \mathcal{S}_\alpha^2$ there exists a pair of roots of unity $(\zeta_1, \zeta_2)$ such that $(\varphi(\alpha \zeta_1), \varphi(\alpha \zeta_2)) = (\beta_1, \beta_2)$. Thus if there exist infinitely many points in $\mathcal{C} \cap \mathcal{S}_\alpha^2$  then one of the components of $(\varphi, \varphi)^{-1}(\mathcal{C})$ is equal to a co-set (by Manin-Mumford). Calling this co-set $\mathcal{C}_\varphi$ and letting $H$ be a coordinate-wise monomial sending $\mathcal{C}_\varphi$ to $\{\beta\}\times \G_m$ for some $\beta \in \overline{\Q}$ we obtain that $\mathcal{C}$ is stratified by $f$. 
Thus we have shown that Theorem \ref{thm1} is a consequence of  \cite[Theorem 1.1]{dynbogomolovcurves}, Manin-Mumford for $\G_m^2$ and Theorem \ref{thm}.\\

For the reduction of Theorem \ref{grandorbits} to Theorem \ref{grandorbit} we can imitate the proof above. However, instead of Manin-Mumford we employ the results in \cite{laurent}.  \\

Now we prove a Lemma that classifies pre-periodic curves that have non-empty intersection with the grand or small orbit of a point $(\alpha,\alpha)$. 
\begin{lem}\label{reductionprep}
	 Let $\mathcal{C}, K, \alpha, \Delta$ be as in Theorem \ref{thm}. If $\mathcal{C}$ is pre-periodic by $(f,f)$ and $\mathcal{C} \cap \mathcal{G}_\alpha^2$ is non-empty then there exist $n,m  \in \mathbb{Z}_{\geq 0}$ such that 
	$$(f^{\circ n}, f^{\circ m})(\mathcal{C}) = \Delta.$$
	Moreover if $\mathcal{C} \cap \mathcal{S}_\alpha^2$ is non-empty then $n = m$. That is, there exists $n \in \mathbb{Z}_{\geq 0}$ such that 
	$$(f^{\circ n}, f^{\circ n})(\mathcal{C}) = \Delta.$$
	\end{lem} 
\begin{proof}  We first replace  $\mathcal{C}$ by an iterate of $\mathcal{C}$, $ (f^{\circ N},f^{\circ N})(\mathcal{C})$ that is periodic by $(f,f)$. By the classification of periodic subvarieties \cite[Theorem 6.2.4]{medvedevscanlon}, there exists a polynomial $g$, commuting with some iterate $f^{\circ M}$ of $f$ such that after possibly permuting $X,Y$ our curve  $\mathcal{C}$ is the the zero-locus of 
	$$Y - g(X) = 0.$$
Now take $(\beta, g(\beta)) \in \mathcal{G}_\alpha^2$. There exist $\ell, k \geq 0$ such that 
$$f^{\circ \ell}(\beta) = f^{\circ k}\circ g(\beta).$$
But then it holds for all $i \in \mathbb{Z}_{\geq 0}$, such that $k + i \in \mathbb{Z}_{\geq 0}M$, that 
$$f^{ \circ \ell + i}(\beta) = g\circ f^{\circ k + i}(\beta).$$
If $k \geq \ell$ then we deduce that $g \circ f^{\circ k - \ell} $ fixes $f^{\circ \ell + i}(\beta)$ for infinitely many $i$. As $\beta$ is not pre-periodic  by $f$  we infer that $g \circ f^{\circ k - \ell} $ is the identity which implies that $k = \ell$ and $g = \text{Id}$. Thus we have proven the Lemma if such $\beta$ exist and in particular if $\mathcal{C} \cap \mathcal{S}_\alpha$ is non-empty. If $k < \ell$ then 
$$g\circ f^{\circ k + i}(\beta) = f^{\circ \ell -k}\circ f^{\circ k + i}(\beta)$$
for infinitely many $i \in \mathbb{Z}_{\geq 0}$ and thus $g = f^{\circ \ell -k}$, which finishes the proof. 
	\end{proof} 
\section{Mahler functions and Böttcher coordinates}
For any field $\mathcal{K}$ (of characteristic 0) and any monic polynomial $f \in \mathcal{K}[X]$ of degree $d \geq 2$ there exists a unique Laurent series $\Psi \in \frac 1X\mathcal{K}[[X]]$ that satisfies 
\begin{align}\label{mahler}
\Psi(X^d) = f(\Psi(X))
\end{align}
and has a pole of residue 1 at 0 \cite{arboreal}. For example for $f = X^d$, we have $\Psi(X) = 1/X$. The equation (\ref{mahler}) defines a Mahler function \cite{NishiokaMahler} which is a well-studied class of functions in transcendental number theory. If $\mathcal{K}$ is a valued field with value $v$ there exists a maximal  $R^v \in \mathbb{R}_{> 0}\cup \{\infty\}$ such that $\Psi$ defines a meromorphic function on 
\begin{align*}
D_{v} = \{z \in \overline{\mathcal{K}}_v; |z|_v < R^v\}
\end{align*}
and moreover there exists an inverse $\Phi \in \frac{1}{X}\mathcal{K}[[\frac 1X]]$ on $\Psi(D_{v}^*)$ (where $D_{v}^* = D_v\setminus \{0\}$) that then satisfies 
\begin{align}
\Phi(f(X)) = \Phi(X)^d.
\end{align}
For non-archimedean places this was (to the best of the authors knowledge) introduced by Ingram in \cite{arboreal} and then further developed in \cite{demarcoetal}.
This inverse $\Phi$ is called the Böttcher coordinate associated to $f$. Our proofs crucially rely on properties of these functions that allow to conjugate the action of $f$ to the more well-understood action of the power map. Another viewpoint is that they provide the analogue of a Tate parametrization. \\

Now we return to a polynomial $f \in K[X]$ for a number field $K$. Our first remark is that at almost all places $v$ of $K$ it holds that $R^v = 1$ \cite[Theorem 6.5]{demarcoetal} and that for non-exceptional polynomials $
R^v \leq 1$ at all places $v$.\\
For $\alpha \in K$ it holds that either $\alpha$ is a pre-periodic point of $f$ or there exists a place $v$ of $K$ such  that $|f^{\circ n}(\alpha)|_v \rightarrow \infty$ \cite[Proof of Theorem 6.1]{polydynamics}. In particular for each $\alpha$ that is not pre-periodic there exists $N$ such that $f^{\circ N}(\alpha) \in \Psi(D_{v}^*)$ at some place $v$.   
Theorem \ref{thm} is invariant under replacing $\alpha$ by an iterate $f^{\circ N}(\alpha)$. Since if we can find a curve $\mathcal{C}$ that intersects $\mathcal{S}_\alpha^2$ in infinitely many points then $(f^{\circ N},f^{\circ N})(\mathcal{C})$ will intersect $\mathcal{S}_{f^{\circ N}(\alpha)}^2$ in infinitely many points. And if $(f^{\circ N},f^{\circ N})(\mathcal{C})$ is special then $\mathcal{C}$ is as well. Thus we may work in $\Psi(D_{v}^*)$ and choose $\alpha$ such that
\begin{align}\label{condalpha}
|\Phi(\alpha)|_v<  (R^v)^d.
\end{align} 
Now enters another dichotomy. We either have that the place $v$ we chose is finite or infinite. Our proof follows a divide and conquer strategy and in the next section we are going to treat the finite places. 

\section{Non-archimedean places}\label{nonarchplaces}
In this section we assume that $f \in K[X]$  for a number field $K$ ($\alpha \in K$) and that $f$ is monic and non-exceptional. 
Assume that $\alpha \in \Psi(D_{v}^*)$ for a finite place $v$ of $K$ and that (\ref{condalpha}) holds. We set  
\begin{align*}
\mathcal{O}_\alpha=\{f^{\circ n}(\alpha); n \in \Z_{\geq 0}\}. 
\end{align*}
and we are going to prove the following uniform version of Theorem \ref{thm} for $\alpha$ whose orbit tends to infinity at a finite place. We need the uniformity in the orbit of $\alpha$ to prove Theorem \ref{grandorbit}. 
 
\begin{prop} \label{prop1} Suppose there is a finite place $v$ of $K$ for which $\alpha \in \Psi(D_{v}^*)$.  Let $\mathcal{C}$ be an irreducible plane curve and suppose that $\mathcal{C}$ is not special. Then there exists a constant $C = C(\alpha, \mathcal{C})$ such that if $(\beta_1, \beta_2) \in \mathcal{C}(\overline{\Q}) \cap \mathcal{S}_{\beta}^2$ for some $\beta \in \mathcal{O}_\alpha$, then $f^{\circ n}(\alpha) = f^{\circ n}(\beta_1) = f^{\circ n}(\beta_2)$ for some $n \leq C$.
\end{prop}
The above proposition directly implies Theorem \ref{thm} if $|f^{\circ n}(\alpha)|_v \rightarrow \infty$ at a finite place $v$ of $K$. \\ 

In what follows in this section we set $D = D_v$ where $D_v$ is as in the previous section,  the maximal the maximal disc on which  $\Psi$ is meromorphic and similarly
\begin{align*}
D^* =\{z \in \overline{K}_v; 0<|z|_v< R^v\}.
\end{align*}

For non-archimedean places the absolute value behaves strikingly simply. It holds that $|\Psi(z)|_v = |z|_v^{-1}$ and that $\Psi$ is a bi-holomorphism 
\begin{align*}
\Psi: D^* \simeq \{z \in \overline{K}_v; |z| > (R^v)^{-1}\}
\end{align*} 
with inverse $\Phi$. \\

We first give a brief informal description of the proof of Proposition \ref{prop1}. We consider the polynomial $P$ whose zero-locus defines $\mathcal{C}$ and fix a point $(\beta_1, \beta_2) \in \mathcal{C} \cap \mathcal{S}_\alpha^2$. Due to our conditions on $\alpha$ we conclude that 
$$P(\Psi(\zeta_1\Phi(\alpha)), \Psi(\zeta_2\Phi(\alpha))) = 0$$
where $\zeta_1, \zeta_2$ are roots of unity. The goal now is to bound the order $N$ of $(\zeta_1, \zeta_2) \in \mathbb{G}_m^2$, which will lead to the bound on $n$ in Proposition \ref{prop1}. In order to achieve this we first re-write $(\zeta_1, \zeta_2)$ as 
$$(\zeta_1, \zeta_2) = (\zeta_1'\zeta_N^{k_1}, \zeta_2'\zeta_N^{k_2})$$
where $\zeta_1', \zeta_2'$ are auxilliary roots of unity of small order (see Lemma \ref{rootsof1}) and $k_1, k_2$ are also small compared to $N$. With the help of this re-parametrization of the roots of unity we  can construct an auxilliary function 
$$a(x) = P(\Psi(\zeta'_1x^{k_1}\Phi(\alpha)), \Psi(\zeta'_2x^{k_2}\Phi(\alpha)))$$
which is a non-archimedean analytic function on an open annulus that has the property that 
$$a(\zeta_N) = P(\beta_1, \beta_2) =  0.$$ 
We then bound the number of zeros of $a$ with the help of Lemma \ref{PJ} and compare this bound to a lower bound for the Galois orbit of $\zeta_N$ over $K_v$ (Lemma \ref{galoislowerbound}). The main bulk of the argument is concerned with bounding the number of zeros of $a$. However the first crucial step is to show that $a(x)$ is not the zero-function. This is where a result from transcendence theory comes in. We rely on work on Mahler functions by Nishioka and Ritt as well as its non-archimedean generalizations by Corvaja and Zannier (see Theorem \ref{transcendence1}). Their work shows that if $\alpha$ is an algebraic number, $\Phi(\alpha)$ is transcendental over $\overline{\mathbb{Q}}$. In combination with our Theorem \ref{transcendence2} this indeed shows that $a$ is not the zero-function unless $\mathcal{C}$ is a pull-back of the diagonal by an iterate of $(f,f)$. \\
In order to work with non-archimedean analytic functions we will need the notion of annuli  and denote by 
$$AI = \{z \in \overline{K}_v;  |z|_v \in I\}$$
for an interval $I \subset \mathbb{R}$. Moreover we denote by 
$$\mathcal{A}I= \left\{ \sum_{n \in \mathbb{Z}}a_n z^n; a_n \in \overline{K}_v, \lim_{|n| \rightarrow \infty}|a_n|_vr^n = 0 \text{ for all } r \in I \right\}$$
the ring of analytic functions on $AI$. For a concise treatment of this see \cite{cherry}. 
\subsection{Functional transcendence}
 In this subsection, let $P\in K[X,Y]\setminus\{0\}$ be an absolutely irreducible polynomial.  We are going to prove a functional transcendence statement that is crucial for the proof of Proposition \ref{prop1}. For the next theorem we do not need to assume that $f$ is non-exceptional and it is easy to see that the theorem holds for exceptional polynomials. \\
It basically states that the image of certain co-sets in $\G_m^2$ by $(\Psi, \Psi)$ are transcendental sets unless the image of the co-set is contained in a special curve. 
\begin{thm}\label{transcendence2}
Let $(k_1,k_2) \in \mathbb{Z}^2 \setminus\{(0,0)\}, \zeta_1, \zeta_2 \in \mu_{\infty}$  and let $a \in D^*$ be transcendental such that $a^{1/d} \in D^*$. Then  
\begin{align*}
P(1/\Psi(\zeta_1ax^{k_1}), 1/\Psi(\zeta_2ax^{k_2})), 
\end{align*}
defines an element of $\mathcal{A}I$, where $I$ is an open interval defined by the conditions
$$AI = \{x \in \overline{K}_v; ax^{k_1}, ax^{k_2} \in D^*\}.$$
If 
$$ P(1/\Psi(\zeta_1ax^{k_1}), 1/\Psi(\zeta_2ax^{k_2})) \in \mathcal{A}I$$
defines the zero-function in $x$ then $k_1k_2(k_1-k_2) = 0$. 
\end{thm}
\begin{proof} After possibly replacing $x$ by $1/x$ and relabelling variables we may assume that $k_1 \geq 0$ and $|k_1| \geq |k_2|$. We may also replace $a$ by $\zeta_1 a$ and set $\zeta= \zeta_2\zeta_1^{-1}$. We pass to the polynomial $P^* = X^{\deg_XP}Y^{\deg_YP}P(1/X,1/Y)$ and note that the Theorem is equivalent to showing that if $P^*(\Psi(ax^{k_1}),\Psi(\zeta ax^{k_2})) = 0$ identically then $(k_1-k_2)k_1k_2 = 0$. We set   $\mathcal{N}^*(x,y) = P^*(\Psi(x), \Psi(y))$ and will work with the expansion of $\mathcal{N}^*$
\begin{align*}
\mathcal{N}^*(x,y)= \sum_{n,m \in \Z_{\geq p}}c_{nm}x^ny^m \in K[[x,y]]
\end{align*}
and note that $\mathcal{N}^*$ is an analytic function on $(D^*)^2$.  Here $p$ is a fixed integer depending on $\mathcal{C}$.  If $k_1k_2 > 0$ we consider the series expansion of $\nu^* = P^*(\Psi(ax^{k_1}), \Psi(\zeta ax^{k_2})) $  
\begin{align*}
\nu^*(x)  = \sum_{k \in \Z}c_kx^k, \text{ }
c_k  = \sum_{k_1n + k_2 m = k}c_{nm}\zeta^{m}a^{n+m}
\end{align*}
and if $k_1 -k_2 \neq 0$ then, for each $\ell \in \Z$,  there is at most one solution $(n,m) \in \Z_{\geq p}^2$ to $k_1n + k_2m = k, n + m = \ell$. Moreover as $k_1, k_2,n,m \in \Z_{\geq p}$, $c_k$ is a polynomial in $a$ with coefficients in $\overline{\Q}$. Thus if $\nu^* = 0$ identically then as $a$ is transcendental $c_{nm} = 0$ for all $n,m \in \Z_{\geq p}$ which contradicts $P \neq 0$.   
Now assume that $k_1k_2 <0$. We can not use the same argument as above as the coefficients $c_k$ are now potentially infinite power series  in $a$. We still have that  for all $x \in A(|a|_v,R) = \{z \in \overline{K_v}; |a|_v<|z|< R \}$ there exists $y\in D^*$ satisfying 
\begin{align} 
P^*(\Psi(x), \Psi(y)) = 0, ~ x^{-k_2}y^{k_1} = \zeta^{k_1}a^{k_1 - k_2} \label{rell1}.
\end{align} 
We  make the substitution $ x_1^d = x, y_1^d = y$ and we find that for all $x_1 \in A^d$, where 
$$A^d = \{z \in A(|a|_v,R); z^d \in A(|a|_v, R)\} = A(|a|_v^{1/d}, R), $$ 
 there exists $y_1 \in D^*$ satisfying
\begin{align}\label{rell2} 
P^*(f\circ \Psi(x_1), f\circ \Psi(y_1)) = 0,  ~ x_1^{-dk_2}y_1^{dk_1} = \zeta_2^{k_1}a^{k_1 - k_2}, x_1, y_1 \in D^*.
\end{align}
Note that since $|a|_v< R^d$, $A^d$ contains infinitely many points. 
 We can take a fibred product of the sets defined by (\ref{rell1}) and (\ref{rell2}) by setting $x_1 = x$ and we obtain that for all $x \in A^d$ there exist $y,z \in D^*$ such that 
\begin{align*} 
P^*(\Psi(x), \Psi(y))& = P^*(f\circ \Psi(x), f\circ \Psi(z)) = 0, \\
x^{-k_2}y^{k_1} & =\zeta^{k_1} a^{k_ 1- k_2},\text{ } x^{-dk_2}z^{dk_1} = \zeta^{k_1}a^{k_1 -k_2}.
\end{align*}
Taking the resultant $R = \text{Res}_X(P^*(X,Y), P^*(f(X), f(Z))) \in K[Y,Z]$ with respect to the variable $X$ we obtain that there are infinitely many $y,z \in D^*$ satisfying 
\begin{align*}
R(\Psi(y),\Psi(z)) = 0, ~ y^{-dk_1}z^{dk_1} = \zeta^{(d-1)k_1}a^{(d-1)(k_1 - k_2)}.
\end{align*} 
Thus if $k_1 -k_2 \neq 0$ there exists a transcendental number $b, |b|_v < 1$ (more precisely $b$ is a $dk_1$-th root of $\zeta_2^{(d-1)k_1}a^{(d-1)(k_1 - k_2)}$) such that 
\begin{align*} 
h(y) = R(\Psi(y), \Psi(by)) = 0,  
\end{align*}
for infinitely many $y \in D^*$ and so $h$ is the zero-function. 
We set $R(\Psi(x),\Psi(y)) = \sum_{n,m \in \Z_{\geq 0}}r_{nm}x^ny^m$ and we find that the coefficients $h_k$ of $h(x) = \sum_{k \in \Z}h_kx^k$ satisfy
\begin{align*} 
h_k = \sum_{n,m \in \Z_{\geq r}, n + m = k}r_{nm}b^{m} =  0
\end{align*}
for some $r \in \Z$ depending on $P^*$. As $b$ is transcendental we obtain that $r_{nm}$ = 0 for all $n,m$ and so $R$ is the 0-polynomial.  It  follows that the partial derivative with respect to $X$, $\partial_XP =0$ identically. For example if not then also $\partial_X Pf' \neq 0 $ and the projection from the variety defined by $P(X ,Y) = P(f(X), f(Z)) =  0$ to the $X$-coordinate would be finite to 1. But $R = 0$ implies that we can find infinitely many solutions $(Y,Z)$ to $P(x_0 ,Y) = P(f(x_0), f(Z)) =  0$ for at least one fixed $x_0 \in \overline{K}$. \\
However if $\partial_X P = 0$ identically then $k_2 = 0$. This concludes the proof. 
\end{proof}

\subsection{Combinatorics}\label{combinatorics}
In this subsection we establish all the combinatorial lemmas that are needed to carry out the proof of Proposition \ref{prop1}. 
For $\underline{a} \in \Z^2$ and $N \geq 2$ we define $S_{\underline{a}} = \Z \underline{a} +N\Z^2$ and for  $c \geq 1$, $B_c = \{x \in \R^2; |x|_\infty \leq c\}$. Here $|\cdot|_\infty$ is the maximum norm on $\R^2$. 
\begin{lem}\label{box1} Let $\underline{a}= (a_1,a_2) \in  \Z^2 \setminus \{0\}$ be such that $\text{gcd}(a_1,a_2, N) = 1$ and $N \geq 17$. Let $c$ be any real satisfying $\frac 34 \leq c \leq  1$. There exist at least $N^{2c-1}/4$ distinct integer vectors $v$ in $S_{\underline{a}} \cap B_{N^{c}}$. 
\end{lem}
\begin{proof} As $\text{gcd}(a_1, a_2,N) = 1$ the reduction of $\underline{a}$ $\mod N$ has exact order $N$ in $\Z^2/N\Z^2$. Thus $S_{\underline{a}} \cap[0,N)^2$ contains $N$ distinct elements. We can cover the box $[0,N)^2$ by at most $([N/N^c] +1)^2 \leq N^{2(1-c)}(1+N^c/N)^2 \leq 4 N^{2(1-c)}$ boxes of side length $N^c$ . Thus there is a box of side length $N^c$ that contains at least $N^{2c -1}/4$ distict integer vectors and translating by one fixed integer vector in this set we also find this many vectors in in $S_{\underline{a}} \cap B_{N^{c}}$.  
\end{proof}
Although the previous Lemma can already be used to give a new proof of Manin-Mumford in $\G_m^n$ \cite{mm} we need to squeeze out more information out of the structure of the set $S_{\underline{a}}$. 
\begin{lem}\label{boom} Let $\underline{a} = (a_1,a_2) \in \Z^2, N \in \Z$ be such that $\text{gcd}(a_1,a_2,N) = 1$ and $N \geq 17$. Let $c$ be real such that $\frac 34 \leq c \leq 1$. There exists an absolute constant $C_1 \geq 1$ such that for every real constant $C \geq 1$ there exists $(k_1,k_2) \in \Z^2, e \in \Z$ such that $\text{gcd}(k_1,k_2) = 1$, $e(k_1,k_2) \in S_{\underline{a}} \cap B_{N^{c}}$ and either $e \leq C_1C^2N^{1-c}$ or $|(k_1,k_2)|_\infty > C$.
\end{lem}
\begin{proof} Suppose that all $w \in S_{\underline{a}} \cap B_{N^c}$ are equal to $w = e(k_1,k_2)$ with $e \in \Z$ and $|(k_1,k_2)|_\infty \leq C$. Then $S_{\underline{a}} \cap B_{N^{c}}$ is contained in the union of at most $(2C +1)^2$ $\Z$-modules $\mathcal{L}_v = \{lv , l \in \Z\}$ of rank 1. However the cardinality of $\mathcal{L}_v \cap B_{N^{c}} \setminus \{0\}$ is bounded by $2N^{c}/|v|_\infty$. Let $m$ be the minimum of $|v|_\infty$ as $v$ runs through $v$ such that $\mathcal{L}_v \cap B_{N^{c}}\cap S_{\underline{a}} \neq \{0\}$.  From Lemma \ref{box1} follows 
\begin{align*}
2(2C+1)^2N^{c}/m \geq N^{2c-1}/4 - 1
\end{align*}
 and so $m \leq C_1C^2 N^{1-c}$. This finishes the proof.  
\end{proof}
We record the following, whose proof is given in \cite{mm}. 
\begin{lem}\label{2} Let $k \geq 2$ be a rational integers and $f = \text{gcd}(N,k)$. There exists an integer $\ell$ such that 
\begin{align*}
k = \ell f \mod N \text{ and } \text{gcd}(\ell,N)= 1. 
\end{align*} 
\end{lem}
Now we apply what we established combinatorially to roots of unity. 
\begin{lem} \label{rootsof1} Fix constants $C \geq 1$ and $3/4 \leq c \leq 1$. Let $(\zeta_1, \zeta_2)$ be a point of order $N \geq 17 $ in $\G_m^2$. There exists a primitive $N$-th root of unity $\zeta$ and integers $k_1,k_2,e$ such that $\text{gcd}(k_1,k_2) = 1$, $|(k_1,k_2)|_\infty \leq N^{c}/e$ and $(\zeta_1, \zeta_2) = (\zeta_e^{(1)}\zeta_N^{k_1}, \zeta_e^{(2)}\zeta_N^{k_2})$ where $\zeta_e^{(1)}, \zeta_e^{(2)}$ are $e$-th roots of 1. Moreover either $e \leq C_1C^2 N^{1-c }$ or $|(k_1,k_2)|_\infty>  C$ for an absolute constant $C_1$ independent of $C$ and $N$. 
\end{lem}
\begin{proof} There exists an integer vector $\underline{a} = (a_1, a_2)$ with the property $\text{gcd}(a_1,a_2,N) = 1$ such that $(\zeta_1, \zeta_2) = (\zeta_N^{a_1}, \zeta_N^{a_2})$ for a primitive $N$-th root of unity $\zeta_N$. Applying Lemma  \ref{boom} we find $(k_1, k_2), e$ as in the statement as well as an integer $k$ such that $k(a_1, a_2) = e(k_1,k_2) \mod N$.   By Lemma \ref{2} we can replace $k$ by $f\ell$ where $f = \text{gcd}(k,N)$ and $\ell$ is invertible $\mod N$. Since $k_1,k_2$ are co-prime we deduce that $f|e$. Denoting by $\ell^*$ the inverse of $\ell \mod N$ we can replace $\zeta_N$ by  $\zeta_N^{\ell^*}$ and obtain the Lemma. 
\end{proof}

\subsection{Zero-counting} 
In this subsection we briefly recall the main counting results over $p$-adic fields. They are consequences of a careful study of Newton polygons and all proofs can be found in the notes \cite{cherry}. We recall that $K_v$  is a completion of $K$ with respect to a non-archimedean place $v$ and we let $g(z) = \sum_{n \in \Z}a_nz^n, a_n \in K_v $ ($g  \in \mathcal{A}[r_1,r_2)$) be a power series defining a non-constant analytic function on an annulus 
$$A[r_1,r_2) = \{z \in \overline{K}_v; r_1 \leq |z|_v < r_2 \}$$
 for real $0<r_1 <r_2 $. 
We define
\begin{align*}
N(g,0,r) = \sum_{z \in   A[r_1,r), g(z) = 0} \log \frac r{|z|},
\end{align*}
(counting with multiplicity) as well as 
\begin{align*}
|g|_r = \sup_n\{|a_n|_vr^n\}.
\end{align*}
We then define
\begin{align*} 
\kappa(f,r) = \inf\{n, |a_n|_vr^n = |g|_r\}.
\end{align*}

We will need the following Lemma \cite[Theorem 2.5.1]{cherry} which is central to our arguments.
\begin{lem}\label{PJ}(Poisson-Jensen)   With $g$ as above it holds that 
\begin{align*}
N(g,0,r) + \kappa(g,r_1)\log r + \log |a_{\kappa(g,r_1)}|_v = \log |g|_r.
\end{align*}
\end{lem}
\subsection{Lowerbounds}
In this short subsection we record the Galois lower bounds that are valid for roots of unity over completions of number fields. As usual $K_v$ is a completion of $K$ with respect to a non-archimedean place $v$ of K.
\begin{lem} \label{galoislowerbound}  Let $N$ be a positive  integer dividing a power of $d$ and $\zeta_N$ a primitive $N$-th root of 1. There exists a positive real number $C_G$ depending on $v,d,K$ such that 
\begin{align*}
[K_v(\zeta_N):K_v] \geq C_G N. 
\end{align*}
\end{lem}
\begin{proof}
This follows from Corollary 6.5 in \cite{polydynamics}. 
\end{proof}
\subsection{Auxilliary function}
In what follows we assume that $\mathcal{C}$ is not special. 
We switch coordinates first. If the curve  $\mathcal{C}$ is given by  affine coordinates $(X_1, X_2)$ we switch to coordinates $(X,Y) = (1/X_1,1/X_2)$ and denote by $P \in \mathcal{O}_K[X,Y]$ the absolutely irreducible polynomial defining $\mathcal{C}$ in these coordinates. We then set 
\begin{align*}
\mathcal{N}(x,y) = P(1/\Psi(x), 1/\Psi(y)) = \sum_{n,m \in \Z_{\geq 0}}a_{nm}x^ny^m \in K[[x,y]].
\end{align*}

For a tuple $\underline{s}= (\zeta_1, \zeta_2, k_1,k_2, \beta)$ where $\zeta_1, \zeta_2$ are roots of unity of order $e$ dividing some power of $d$ and $k_1,k_2$ are co-prime rational integers satisfying $k_1 > 0, k_1 \geq k_2$ and $\beta \in \mathcal{O}_\alpha$ we define $\nu = \nu_{\underline{s}}$ by
\begin{align}\label{nu}
\nu(x) = \mathcal{N}(\zeta_1\phi x^{k_1}, \zeta_2\phi x^{k_2}) = \sum_{k \in \Z}b_kx^k,
\end{align}
where $\phi = \Phi(\beta)$ and $b_k$ is given by 
\begin{align}\label{bk}
b_k = \sum_{k_1n + k_2m = k}a_{nm}\phi^{n+m}\zeta_1^{n}\zeta_2^{m}. 
\end{align}
We note that $b_k$ does not necessarily lie in $\overline{K}$ but in $\overline{K}_v$. Moreover the function $\nu$ is an analytic function on the annulus $A[1,r)$  where $r = R/|\phi|_v^{1/{|(k_1,k_2)|_\infty}}$.  Note that we have $|\nu|_v \leq  1$. \\

Now we proceed to bounding $r$ and $|\nu|_1$ from below.  We first prove these bounds up to finitely many possibilities for $\beta \in \mathcal{O}_\alpha$. 

\begin{lem}\label{bounded} There exists a constant $C_{O}$ depending only on $\mathcal{C}$ such that if $|\phi|_v < C_O$ and $k_1 \neq k_2$, then 
\begin{align*}
-\log |\nu|_1\leq -C_2\log|\phi|_v + C_3 
\end{align*}
for positive constants $C_2, C_3$ depending only on $\mathcal{C}$. 
\end{lem}
\begin{proof} 
Recall our convention that $k_1 > 0$ and $k_1 \geq k_2$. Let $(n_0,m_0)$ be a pair of positive integers such that $a_{nm}\neq 0$ and $n +m$ is minimal with this property. Note that there are at most finitely many such pairs.  Let $k_0 = k_1n_0 + k_2m_0$ and first suppose there exists another pair of positive integers $(n_1,m_1) \neq (n_0,m_0)$ such that $k_1n_1 + k_2m_1 = k_0$ and $|a_{n_0m_0}\phi^{n_0 + m_0}|_v \leq |a_{n_1m_1}\phi^{n_1 + m_1}|_v$.  Then $|a_{n_0m_0}|_v \leq |\phi|^{k_1 - k_2}_v$ and by our assumption $k_1- k_2 > 0$. Thus if $|\phi|_v < |a_{n_0m_0}|$ no such pair $(n_1,n_2)$ exists and it follows  that $|b_{k_0}|_v = |a_{n_0m_0}\phi^{n_0 + m_0}|_v$. As $|b_{k_0}| \leq |\nu|_1$ the Lemma follows.  
\end{proof}
From this point onward we may assume that $\beta = f^{\circ m}(\alpha)$ for $m$ bounded by a constant only depending on  $\mathcal{C}$. However we will indicate when we will make use of this fact. 

\begin{lem}\label{lower1} There exists a constant $C_{\text{lower}} = C_{\text{lower}}(\alpha,\mathcal{C}) \geq 2$ such that if $|(k_1,k_2)|_\infty \geq C_{\text{lower}}$ then 
\begin{align*}
-\log |\nu|_1 \leq -C_2\log |\phi|_v + C_3
\end{align*}
 for positive constants $C_2, C_3$ depending only on $\mathcal{C}$.  
\end{lem}
\begin{proof}
Let $(n_0,m_0)$ be as in the previous Lemma a pair of positive integers such that $a_{nm}\neq 0$ and $n +m$ is minimal with this property and again suppose that there exists another pair of positive integers $(n_1,m_1) \neq (n_0,m_0)$ such that $k_1n_1 + k_2m_1 = k_0$ and $|a_{n_0m_0}\phi^{n_0 + m_0}|_v \leq |a_{n_1m_1}\phi^{n_1 + m_1}|_v$. \\
We first prove that if $k_1 \neq k_2$ then $|k_1 - k_2| \geq c_1|(k_1,k_2)|_\infty $ where $c_1 >0$ is a constant only depending on $\mathcal{C}$. \\
If $k_1k_2 <0$ this is trivial and we can set $c_1 =1$. If $k_1k_2 >0$ we note that $(n_1 -n_0, m_1 -m_0) \in \Z(-k_2,k_1)$ and so either $n_1 -n_0 \leq 0$ or $m_1 - m_0 \leq 0$. If $n_1 - n_0 \leq 0$ then $n_1 -n_0 \leq -k_2$ and as  $k_2 >0$ we deduce $k_2 \leq n_0$. Similarly if $m_1 -m_0 \leq 0$ we deduce that $k_1 \leq m_0$. In either case we have that $|2(k_1-k_2)(n_0,m_0)|_\infty \geq |(k_1, k_2)|_\infty$. \\
We continue by noting as in the previous Lemma that $|\phi|_v^{k_1-k_2} \geq |a_{n_0m_0}|_v$. Thus if $k_1 -k_2 > \log|a_{n_0m_0}|_v/\log |\phi|_v$ then we can not find $(n_1,n_2)$ as above.  We deduce from the non-archimedean tri-angle inequality and what we just proved that if $|(k_1,k_2)|_\infty > c_1^{-1}\log|a_{n_0m_0}|_v/\log |\phi|_v$ then $|b_{k_0}|_v = |a_{n_0m_0} \phi^{n_0 + m_0}|_v$ and we are finished.
\end{proof}
In the next two lemmas we are going to use the fact that we can concentrate on a finite subset of $\mathcal{O}_\alpha$ which is reflected in a bound for the height. 
\begin{lem} \label{upper1}Suppose that $|(k_1,k_2)|_\infty \leq C_1$ and $h(\beta) \leq C_1$ for some positive constant $C_1$. Then either 
\begin{align} \label{ineq1}
-\log |\nu|_1 \leq C_2e^{\log d/\log 2}
\end{align}
for $C_2 = C_2(C_1,\mathcal{C}, \alpha)$ or $\ord(\zeta_1^{-k_2}\zeta_2^{k_1}) \leq C_3(\alpha,\mathcal{C})$. 
\end{lem}
\begin{proof} (First note that $e \neq \exp(1)$). Set $x_0 = \zeta_1^{-1/k_1}$ where we pick a $k_1$-th root such that the order of $x_0$ still divides a power of $d$. This can be done by writing $k_1 =k\tilde{k}$ with $k$ dividing a power of $d$ and $\tilde{k}$ co-prime to $d$. We can write $\zeta_1 = \zeta_2^{\tilde{k}}$ with $\zeta_2$ having the same order as $\zeta_1$ and then pick any $k$-th root of $\zeta_2$. 
   We obtain $\nu(x_0) = P(\beta, \Psi(\zeta_2\zeta_1^{-k_2/k_1}\phi))$. Now if $\nu(x_0) = 0$ then the degree $[K_v(\zeta_2\zeta_1^{-k_2/k_1}):K_v] \leq \deg(\mathcal{C})$ since otherwise we can let $\text{Gal}(\overline{K_v}/K_v)$ act on the point  $(\beta, \Psi(\zeta_2\zeta_1^{-k_2/k_1}\phi)$ and find more then $\deg(\mathcal{C})$ points on $\mathcal{C}$ with $X$-coordinate equal to $\beta$. This would imply that $\mathcal{C}$ is special.  Thus $\ord(\zeta_2^{k_1}\zeta_1^{-k_2})$ is bounded as in the statement. \\
Now suppose that $\nu(x_0) \neq 0$. We note that the point $(\beta, \Psi(\zeta_2\zeta_1^{-k_2/k_1}\phi))$ is algebraic of degree at most $|e(k_1,k_2)|_\infty^{\log d/\log 2}$ over $K$ as $|e(k_1,k_2)|_\infty$ is a bound on the order of $\zeta_2\zeta_1^{-k_2/k_1}$ and  so 
\begin{align*}
f^{\circ n}(\beta) = f^{\circ n}(\Psi(\zeta_2\zeta_1^{-k_2/k_1}\phi)).
\end{align*}
with $ d^n \leq |e(k_1,k_2 )|_\infty^{\log d/\log 2} $.
 So $H(\nu(x_0))\leq c_1c_2^{e^{\frac{\log d}{\log 2}}}$ for constants $c_1,c_2$ depending only on $\mathcal{C}$ and $\alpha$ and we deduce the inequality (\ref{ineq1}) from the standard $|y_0|_v \geq H(y_0)^{-[\Q(y_0):\Q]}$ that is valid for any algebraic point $y_0$. 
\end{proof}
\begin{lem}\label{upper2} Suppose that $\ord(\zeta_1^{-k_2}\zeta_2^{k_1}) \leq C_1, h(\beta) \leq C_1$ and $|(k_1, k_2)|_\infty \leq C_1$ for some constant $C_1$ and that $\nu$ is not the zero function. Then 
\begin{align*} 
-\log |\nu|_1 \leq -C_2
\end{align*}
for a positive constants $C_2 = C_2(C_1, \mathcal{C},\alpha)$. 
\end{lem}

\begin{proof} 
We first note that there are only a finite number of possibilities for $\beta$ as the height of $\beta$ is bounded. There are also only a finite number of possibilities for $k_1,k_2$. For each $k$ we can choose $(n_k,m_k)$ such that $k_1n_k + k_2m_k = k$ and factor out $\zeta_1^{n_k}\zeta_2^{m_k}$. Note that for each $k$ the choice of  $(n_k,m_k)$ only depends on $k_1,k_2$. Now we see that $b_k = \zeta_1^{n_k}\zeta_2^{m_k}\tilde{b}_k$ for  $\tilde{b}_k$ only depending on $\beta, k_1,k_2$ and $\zeta_1^{-k_2}\zeta_2^{k_1}$. It follows that if $\beta, k_1,k_2, \zeta_1^{-k_2}\zeta_2^{k_1}$ vary in a finite set, the sequence $(|b_k|_v)_{k\in \Z}$ varies in  a finite set. Thus $|\nu_{\underline{s}}|_1$ varies in a finite set and as $|\nu|_1 >0$ we get an absolute lower bound for $|\nu|_1$. 
\end{proof}
 \begin{lem} \label{lowerbound}
Suppose that $\nu$ as above (\ref{nu}) is not the zero-function. It holds that 
\begin{align*}
-\kappa (\nu,1) & \leq |(k_1,k_2)|_\infty \log |\nu|_1/\log(|\phi|_v).
\end{align*} 
\end{lem} 
\begin{proof} If $k_1k_2 > 0$ then $b_k = 0$ for $k<0$ and the statement becomes trivial. Thus we may and will assume that $k_1k_2<0$. From the expansion of $b_k$ (\ref{bk}) we deduce that  $|b_k|_v \leq |\phi|_v^{\ell_k}$ where $\ell_k$ is defined by $\ell_k = \min_{m,n \in \Z_{\geq 0}}\{m + n; k_1m + k_2n = k\}$. Since $|k| \leq |(k_1,k_2)|_\infty \ell_k$ we see that for $\kappa = \kappa(\nu,1)$ holds $|\nu|_1 \leq |\phi|_v^{|\kappa|/|(k_1,k_2)|_\infty}$ and the Lemma follows.
\end{proof}

\subsection{Transcendence} In this  subsection we record a transcendence result that belongs to the circle of classical transcendence results on Mahler functions. It is crucial for the application of Theorem \ref{transcendence2} to the proof of Theorem \ref{thm}. 

As a direct consequence of work of Nishioka or Corvaja and Zannier we obtain the following theorem. 
\begin{thm} \label{transcendence1} (Nishioka) Let $f \in K[X]$ be a  non-exceptional polynomial. 
For $\alpha \in D_{v}^*\cap \overline{\Q}$, $\Psi(\alpha) $ is transcendental.  
\end{thm}
\begin{proof} From \cite[Theorem 1]{transcboettcher} follows that $\Psi$ is not an algebraic function. Now for $v$ an archimedean place the present lemma follows from a theorem of Nishioka (see for example \cite[Corollary of Theorem 1.5.1]{NishiokaMahler}). For a non-archimedean place we could perhaps trace the steps of Nishioka's proof and prove it in the non-archimedean case. However it also follows almost directly from a result of Corvaja and Zannier \cite{corvaja_zannier_applications}.\\
Namely if $\Psi(\alpha)$ is algebraic we have that $\Psi(\alpha^{d^n}) = f^{\circ n}(\Psi(\alpha))$ lies in the same field. We then apply Theorem 1 of their paper applied to the function $\Psi(R^v z)$ and $z_{d^n} = (R^v)^{-1}\alpha^{d^n}$. 
\end{proof}
\subsection{Proof of Proposition \ref{prop1}}

Recall that we assume that $|\Phi(\alpha)|_v  < (R^v)^d$. From Lemma \ref{transcendence1} follows that $\Phi(\beta)$ is a transcendental number for all $\beta \in \mathcal{O}_\alpha$. We start with a solution $(\beta_1, \beta_2) \in \mathcal{S}_\beta^2 \cap \mathcal{C}$ and let $n$ be minimal such that $f^{\circ n}(\beta_1) = f^{\circ n}(\beta_2) = f^{\circ n}(\beta)$.  We then have that 
\begin{align*}
\beta_1 = \Psi(\zeta_1 \phi), \hspace{2mm} \beta_2 =  \Psi(\zeta_2 \phi), \hspace{2mm} \phi = \Phi(\beta)
\end{align*}
where $(\zeta_1, \zeta_2) \in \G_m^2$ is of order $N$ dividing $d^n$ and it holds that $2^n \leq N$. Thus we may assume that $N \geq 17$, since we are finished otherwise. Let $C_{\text{lower}}$ be the constant from Lemma \ref{lower1}. \\

We set 
$
c = 1- \frac{\log 2}{4\log d}
$
 and  employing Lemma \ref{rootsof1} we set $(\zeta_1, \zeta_2) = (\zeta_e^{(1)}\zeta_N^{k_1}, \zeta_e^{(2)}\zeta_N^{k_2})$, where $\zeta_N$ is a primitive $N$-th root of unity,  $\zeta_e^{(1)}, \zeta_e^{(2)}$ are $e$-th roots of unity (with $e$ dividing $N$), and $|(k_1, k_2)|_\infty \leq N^{c}/e$. Moreover either $|(k_1, k_2)|_\infty > C_{\text{lower}}$  or $e \leq C_\text{upper}N^{ 1 - c}$ for a constant $C_\text{upper}$ depending only on $\mathcal{C}, \alpha$.  
We set $\nu = \nu_{\underline{s}}$, where $\underline{s} = (\zeta_e^{(1)}, \zeta_e^{(2)}, k_1,k_2)$.  \\

If $k_1 = 0$ then $\text{Gal}(K_v (\zeta_N)/K_v(\zeta_e^{(1)}))$ fixes $\beta_1$ while the Galois orbit of $\beta_2$ is bounded from below by $(C_G/C_{upper})N^{c}$. Hence $N$ and the also $n$ are bounded unless $\mathcal{C}$ is special. Similarly if $k_2 = 0$. Now if $k_1 =k_2$ then $f^{\circ n}(\beta_1) = f^{\circ n}(\beta_2)$ for $n$ such that $d^n \leq e^{\log d/\log2} \leq C_d N^{1/4}$ ($C_d = C_{\text{upper}}^{\log d/\log2}$). We deduce that the Galois orbit of $(\beta_1, \beta_2)$ is contained in the intersection of $\mathcal{C}$ and the curve of degree $d^n$ defined by $f^{\circ n}(X) -f^{\circ n}(Y) = 0$. Now the Galois orbit is bounded from below $C_GN$ and thus it follows from Bézout's theorem that $N$ is bounded unless $\mathcal{C}$ is a component of $f^{\circ n}(X) - f^{\circ n}(Y) = 0$ and thus special. \\

From this point on we assume that $\mathcal{C}$ is not special and $k_1k_2(k_1-k_2) \neq 0$. From Theorem \ref{transcendence1} and Theorem \ref{transcendence2}  follows that $\nu $ is not the zero-function. Moreover the function $\nu$ is defined on the annululs $A[1, r)$ with $r \geq c_v^{1/|(k_1,k_2)|_\infty} $, where $c_v = R/|\phi|_v > 1$. \\
We first note that 
\begin{align*}
|(k_1,k_2)|_\infty\log r \geq \log R - \log(|\phi|_v) \geq -C_1\log(|\phi|_v)
\end{align*}
for a constant $C_1$ depending only on $\alpha$ and $v$. Moreover from Lemma \ref{lowerbound} follows that $-\kappa(\nu,1) \leq |(k_1,k_2)|_\infty \log|\nu|_1/\log(|\phi|_v)$. 
From Lemma \ref{PJ} follows that
\begin{align*}
 Z = |\{x \in \overline{K_v}; \nu(x) = 0, |x|_v = 1\}| \leq -\kappa(\nu,1) - \frac{\log|\nu|_1}{\log r}
\end{align*}
and thus 
\begin{align} \label{Z}
 Z \leq 2 C_1|(k_1,k_2)|_\infty \log|\nu|_1/\log(|\phi|_v) . 
\end{align}
As $|f^{\circ m}(\alpha)|_v = |\Phi(\alpha)|_v^{d^m}$ it holds that if $|\Phi(\beta)|_v \geq  C_O$ (with $C_O$ as in Lemma \ref{bounded}) then $m$ is bounded and thus also $h(\beta)$ is bounded. Thus there exists a constant $C_4$ depending only on $\alpha, \mathcal{C}$ such that either
\begin{align}\label{ineqproof}
-\log|\nu|_1 \leq -C_2\log |\phi|_v + C_3,
\end{align}
for positive constants $C_2,C_3$ depending only on $\alpha,\mathcal{C} $, or $h(\beta) \leq C_4$. We assume from now on that $h(\beta ) \leq C_4$. 
If also $|(k_1,k_2)|_\infty \leq C_{\text{lower}}$  then it follows from Lemma \ref{upper1} and \ref{upper2} that either 
\begin{align}\label{Z1}
 Z \leq C_5e^{\log d/\log 2} \leq C_6N^{\frac 14} 
\end{align}
for a constant $C_6$ depending only on $\alpha, \mathcal{C}$ or $-\log|\nu|_1 \leq -C_2$. 
However if $|(k_1,k_2)|_\infty > C_{\text{lower}}$ then it follows from Lemma \ref{lower1} that again the inequality (\ref{ineqproof}) holds. However from either inequality (\ref{Z1}) or (\ref{ineqproof}) follows that 
\begin{align*} 
 Z \leq C_7 N^{c} /e
\end{align*}
for a constant $C_7$ depending only on $\alpha, \mathcal{C}$. 
On the other,  Lemma \ref{galoislowerbound} gives us that 
\begin{align*} 
 Z \geq  C_8N/e
\end{align*}
for a constant $C_8$ depending only on $\alpha, \mathcal{C}$. As $c <1$, combining the last two inequalities yields $N \leq C_9$ for $C_9$ depending only on $\alpha$ and $\mathcal{C}$. As $N\geq 2^n$ we obtain a bound on $n$.

\section{Connected Julia sets}\label{connectedjulia}

In this section we assume that $f \in \C[X]\setminus \{0\}$ (thus $\mathcal{K} = \C$) of degree $d \geq 2$ and that $f$ is monic. We first recall the following invariants of the dynamical system
\begin{align*}
B_\infty & = \{z \in \C; |f^{\circ n}(z)| \rightarrow \infty, n\rightarrow \infty \},\\ 
 J & = \{z \in \C; \{f^{\circ n}\}_{n \in \Z_{\geq 0}} \text{ is not a normal family in any neighbourhood of } z\}
\end{align*}
where $B_\infty$ is called the basin of infinity and $J$ the Julia set of $f$.  These sets are invariants of the dynamical system (hence equal for all iterates of $f$) and for polynomials are related by $\partial B_\infty = J$. 
In this section, unless we indicate otherwise, we assume that $J$ is connected. This implies that $\Psi$ is meromorphic  on a disc of radius 1 \cite{milnor} and we set 
\begin{align*}
D = \{z \in \C; |z| < 1\}, D^* = D\setminus \{0\}
\end{align*} 
in this section. We moreover have that $\Psi$ is a bi-holomorphism
\begin{align*}
\Psi: D^* \simeq B_\infty
\end{align*} 
(and thus is a Riemann mapping of $B_\infty \cup \{\infty\}$). We will make use of real analytic properties of analytic functions so it is convenient to introduce some notation. For any Laurent series $F \in \C((X_1,\dots, X_n))$ we set $F^h$ to be the Laurent series obtained by complex conjugating the coefficients of $F$.   For example if $F$ is univariate and defines a meromorphic function in some open disc about 0 then  $F^h(z) = \overline{F(\overline{z})}$ is the ususal Schwarz reflection (with $\overline{\cdot }$ denoting complex conjugation ). 
 In what follows we define for real $s \leq 1$, $S^s = \{z \in \C; |z|= s\}$. It is straightforward to see that for $r>0$
\begin{align*}
\Psi(S^{\exp(-r)}) = L_{r}.
\end{align*}
 In this section we consider an irreducible polynomial $P \in \C[X,Y]\setminus \{0\}$ defining the curve $\mathcal{C}$. We assume that $\partial_XP\partial_YP \neq 0$. Moreover we assume that $\mathcal{C}$ has infinite intersection with $L_r^2$ for some $r > 0$ and we let  $\phi = \exp(-r)$.  Since $L_r, \mathcal{C}$ are real analytic sets we deduce that $\mathcal{C}\cap L_r^2$ contains an analytic connected arc. Thus we obtain an analytic curve $(a_1,a_2):(0,1) \rightarrow (S^1)^2$ such that 
\begin{align*}
P(\Psi(a_1\phi), \Psi(a_2\phi)) = P^h(\Psi^h(a_1^h\phi), \Psi^h(a_2^h\phi)) = 0, t \in (0,1).
\end{align*}
Since $a_1^h = 1/a_1, a_2^h = 1/a_2$ we obtain that there is an analytic arc $\mathcal{A}\subset S^\phi\times S^\phi$ such that 
\begin{align}\label{arc}
P(\Psi(x), \Psi(y)) = P^h(\Psi^h(\phi^2/x), \Psi^h(\phi^2/y)), \text{ for all } (x,y) \in \mathcal{A}. 
\end{align}
As our proofs rely on analytic continuation we are going to use paths frequently. For a path $\gamma: I \rightarrow X$ into some complex analytic set $X$ we denote by $|\gamma|$ its support $\gamma(I)$. Here $I$ is an interval (not necessarily open or closed) and we always assume that $\gamma$ is piecewise smooth. For a set $U \subset \C$ we denote by $\partial U, \overline{U}$ its boundary and closure respectively. Further by $U^h$ its complex conjugate. \\

We give a brief summary of the proof strategy in this section. Our proof goes by constructing a pseudo-involution $\mu$ on the set 
\begin{align*}
\mathcal{F} = \Psi(A), A = \{z \in \C; \phi^2<|z|<1\}.
\end{align*}
By a pseudo-involution we mean an analytic map $i:\mathcal{F} \rightarrow \mathcal{F}^h$ such that $i\circ i^h = \text{id}$.  
Using (\ref{arc}) we show that $\mathcal{C}$ intersects $\partial \mathcal{F}^2$ in a relatively open subset unless $f$ is exceptional. More specifically, since $\partial \mathcal{F} = \Psi(S^{\phi^2})\times J$ we will show then that either  the smooth boundary component  $\Psi(S^{\phi^2})$ parametrizes a relatively open subset of $J$ which implies again that $f$ is exceptional  or  $\mathcal{C}$ intersects $\Psi(S^{\phi^2})\times \Psi(S^{\phi^2})$ in infinitely many points. We can then proceed inductively to construct infinitely many pseudo-involutions. This construction allows us to use Schwarz's lemma to show that the polynomial $P$ defines a ``rotation" of $B_\infty$.  We then use a theorem of Levin to show that this rotation is in fact rational. However in order to show that this rotation defines a curve as defined in Theorem \ref{equipotential}  we also need to employ the classification  of periodic curves by Medvedev and Scanlon.  \\

Our starting point for a proof of Theorem \ref{equipotential} is an observation on complex algebraic curves. 

\begin{lem}\label{continuation} Let $U_1,U_2 \subset \C$ be bounded open and connected. Let $C$ be an irreducible complex plane curve whose projection to each coordinate is a dominant map. We denote the canonical projections to the affine coordinates $\pi_1,\pi_2$. Suppose that $C(\C) \cap U_1\times U_2$ is non-empty. 
If  $C(\C)\cap (\partial U_1\times U_2\cup U_1\times \partial U_2)$ is empty then each path $|\gamma| \subset U_1$ or $|\gamma| \subset U_2$ can be lifted to a path in $C\cap U_1\times U_2$ (via $\pi_1$ respectively $\pi_2$).  Moreover the maps $\pi_i :C \cap \overline{U_1}\times \overline{U_2} \rightarrow \overline{U}_i, i = 1,2$ are surjective and $\pi_1^{-1}(\partial U_1) = \pi_2^{-1}(\partial U_2) \subset \partial U_1 \times \partial U_2$.  \\
If  the boundaries $\partial U_1, \partial U_2$ do not contain isolated points and $C \cap(\partial U_1\times U_2\cup U_1\times \partial U_2) $ is non-empty then it contains infinitely many points. 
\end{lem}

\begin{proof} For the first part of the statement, assume that $C \cap U_1\times U_2$ is non-empty and let $\gamma:[0,1]\rightarrow U_1$ be a path such that $\gamma(0) \in \pi_1(C\cap U_1\times U_2)$. As $C$ is connected we can lift $\gamma$ to a path $\tilde{\gamma}:[0,1] \rightarrow C$ such that $\pi_2(\tilde{\gamma}(0)) \in U_2$ and  as $C\cap(\partial U_1\times U_2\cup U_1\times \partial U_2)$ is empty, $|\pi_2\circ \tilde{\gamma} | \subset U_2 $ by the intermediate value theorem. We can repeat the arguments for a path in $U_2$ and obtain the first part of the statement. As $U_1$ is path connected we can find for each boundary point of $U_1$ a path $\gamma$ as above with the endpoint being that boundary point. This proves that $\pi_1, \pi_2$ surject onto the boundaries. \\
If $(x_0,y_0) \in C \cap\partial U_1\times U_2$ then as $\partial U_1$ does not contain isolated points, each neighbourhood of $x_0$ contains infinitely many points of $\partial U_1$. Employing the implicit function theorem or Puiseux's theorem we can find a non-constant analytic function $h$ defined on a neighbourhood of 0  and satisfying $h(0) = x_0$ and a positive integer $N$ such that $(h(x), y_0 + x^N) \in C(\C)$ for $x $. We can find infinitely many $x$ in any neighbourhood of $0$ such that $y_0 + x^N \in \partial U_1$ and $h(x) \in U_1$. This proves  the claim. 
\end{proof}

We define an analytic function $\mu$ on $\mathcal{F}$ by 

\begin{align*}
\mu(z) = \Psi^h(\phi^2/\Phi(z)).
\end{align*} 
and we call $\mu$  a pseudo-involution as $\mu:\mathcal{F} \rightarrow \mathcal{F}^h$ is bijective with $\mu^h\circ \mu = \mu \circ \mu^h = \text{id}$.  
We can reformulate the relations (\ref{arc}) in terms of $\mu$ as
\begin{align}\label{arc2}
P(x,y) = P^h(\mu(x), \mu(y)) = 0, \text{ for all } (x,y) \in \mathcal{B}
\end{align}
where $\mathcal{B} = \Psi(\mathcal{A})$. We also note that $J(f^h) = J^h, B_\infty(f^h) = B_\infty^h$. 
\begin{lem} \label{limits} The function $\mu$ has the property that for each sequence  $\{x_n\}_{n} \subset F$ such that $\lim_n x_n = x \in J$ it holds that  every limit point $y$ of $\{\mu(x_n)\}$ satisfies $y \in \Psi^h(S^{\phi^2})$. Similarly  for each $\{x_n\}_n \subset F$ such that $\lim_nx_n = x \in  \Psi(S^{\phi^2})$ any limit point  $y$ of $\{ \mu(x_n)\}_n $ satisfies $y \in J^h$. 
\end{lem}
\begin{proof} This follows from the fact that $J$ is the boundary of $B_\infty$  and that $B_\infty = \Psi(D^*)$. 
\end{proof}

We will use the properties of $\mu$ that are outlined in Lemma \ref{limits} to deduce that the existence of such an arc $\mathcal{B}$ (\ref{arc2}) implies that $f$ has to be exceptional. Our strategy is to reduce this to the following theorem of Fatou that he proved in his seminal papers.  
\begin{thm}(\cite[p.250]{fatou}) \label{fatou} If a relatively open subset of the Julia set of a rational map $f$ is  an analytic arc then $f$ is conjugate to an exceptional polynomial. 
\end{thm}
By an analytic arc we mean the image of an analyic map $\gamma:(0,1) \rightarrow \mathbb{C}$. 
We recall that $P$ is the polynomial defining $\mathcal{C}$ and we set $C_0$ to be the connected component of $\mathcal{C} \cap \mathcal{F}^2$ containing $\mathcal{B}$. 

\begin{lem}\label{cont1} Suppose there exists $(x_0,y_0) \in \overline{C_0}\cap \left( \partial \mathcal{F} \times \mathcal{F}\cup \mathcal{F}\times \partial \mathcal{F}\right)$ such that $P(x_0,y_0) = 0$. Then we can continue $\mu$ analytically to an open disc centred at a point $x \in \Psi(S^{\phi^2})$.  
\end{lem}
\begin{proof} We can assume that $(x_0,y_0) \in \partial \mathcal{F} \times \mathcal{F}$ and from the relation (\ref{arc2}) as well as Lemma \ref{limits} follows that we may assume that $(x_0,y_0) \in \Psi(S^{\phi^2})\times \mathcal{F}$ (up to replacing $P$ by $P^h$ and $C_0$ by $(\mu, \mu)(C_0)$.  As $\partial \mathcal{F}$ does not contain isolated points we can pick $(x_0,y_0)$ such that $\partial_Y P(x_0,y_0)\partial_Y P(x,\mu(y_0)) \neq 0$ for all $x$ such that $P(x,\mu(y_0))= 0$. We can pick an open disc $D$ centred at $x_0$ such that $D \setminus \Psi(S^{\phi^2}) = D_1\cup D_2$ with $D_1, D_2$ open and such that $D_1 \subset \mathcal{F}, D_2 \subset \C\setminus \mathcal{F}$. Moreover such that there exists an analytic function $h$ on $D$ with $h(x_0) = y_0$ and $P(x,h(x)) = 0$ for $x \in D$. We then have that 
\begin{align*}
P(x,h(x)) = P^h(\mu(x), \mu\circ h(x)) = 0, x \in D_1. 
\end{align*} 
As $y_0 \in \mathcal{F}$, $\mu(y_0) = \mu\circ h(x_0) \in \mathcal{F}$. We set $\mu(x_0)$ to be equal to a limit of points $\mu(x_n)$ where $(x_n)_n$ is a sequence of points in $D_1$ converging to $x_0$.  As $\partial_YP(x,\mu(y_0))\neq 0$ for all $x$ satisfying $P(x,y) = 0$ there exists an analytic function $\tilde{h}$ on an open disc $\tilde{D}$ centred at $\mu(y_0)$ satisfying $P(\tilde{h}(y), y) = 0$ for all $y \in \tilde{D}$ and $\tilde{h}(\mu(y_0)) = \mu(x_0)$. We can shrink $D$ such that $\mu \circ h(D)\subset \tilde{D}$.   From the uniqueness of $\tilde{h}$ follows that 
\begin{align*}
\tilde{h}\circ \mu \circ h(x) = \mu(x)
\end{align*} 
for all $x \in D_1$ and so $\tilde{h}\circ \mu \circ h(x) $ is a continuation of $\mu$ to $D$. 
\end{proof}

\begin{lem} Suppose we can continue $\mu$ analytically to an open disc  centred at $\Psi(S^{\phi^2})$. Then we can find a small disc $D$ centred at $x_0 \in \Psi(S^{\phi^2})$ with the following properties. It holds that $D\setminus \Psi(S^{\phi^2}) = D_1\cup D_2$ with $D_1,D_2$ open such that  $D_1 \subset \mathcal{F}, D_2 \subset \Psi(A(0,\phi^2))$. Moreover $\mu$ continues to $D$ and $\mu^h$ to $\mu(D)$ and it holds that $\mu^h\circ \mu(z) = z $ on $\mathcal{F}\cup D$.   
\end{lem}
\begin{proof} The first part of the Lemma concerning the choice of $D$ follows from the fact that $\Psi(S^{\phi^2})$ is an analytic arc we can further choose $D$ such that $\mu'(z) \neq 0$ on $D$ and $\mu(D)$ is simply connected. Thus we can define $\mu^{-1}$ on $  \mu(D)$ and as it holds that 
\begin{align*}
\mu^h(z) = \mu^{-1}(z), z \in \mu(D_1)
\end{align*}
we can continue $\mu^h$ to $\mu(D)$ and so it holds that $\mu^h\circ \mu(z) = z$ on $\mathcal{F}\cup D$. 
\end{proof}
\begin{corollary}\label{cor1} Suppose we can continue $\mu$ analytically to a disc $D$ centred at $x_0 \in \Psi(S^{\phi^2})$. Then 
$f$ is exceptional.
\end{corollary}
\begin{proof} We can assume that $D$ is as in the previous Lemma and we set $D_1, D_2$ as in the previous Lemma. We first note that $\mu$ is injective on $\mathcal{F} \cup D$ as it has a global inverse. Thus $\mu(D_1) \subset \mathcal{F}^h$ and $\mu(D_2) \subset \C\setminus B_\infty^h$ as $\mu(\mathcal{F}) = \mathcal{F}^h$ and $\mu$ is injective. It follows that $\mu(\Psi(S^{\phi^2})\cap D) = J^h\cap \mu(D)$ and the claim follows from Fatou's theorem (Theorem \ref{fatou}). 
\end{proof}
We turn to the case for which $\overline{C_0} \cap \partial \mathcal{F} \times \mathcal{F}\cup \mathcal{F}\times \partial \mathcal{F}$  is empty. With the following Lemma we discard one case right away.
\begin{lem} \label{cont2} Suppose that $\overline{C_0} \cap \partial \mathcal{F} \times \mathcal{F}\cup \mathcal{F}\times \partial \mathcal{F}$ as well as $\overline{C_0}\cap \left( \Psi(S^{\phi^2})^2\cup J^2\right) $ is empty. Then $f$ is exceptional. 
\end{lem}
\begin{proof} By Lemma \ref{continuation} it holds that the projections $\pi_1, \pi_2$ from $\overline{C_0} \cap \overline{\mathcal{F}}^2$ to $\partial \mathcal{F}$ are surjective. As $\partial \mathcal{F} = \Psi(S^{\phi^2})\times J\cup J \times \Psi(S^{\phi^2})$ and our assumption it holds that $\pi_2(\pi_{1}^{-1}(\Psi(S^{\phi^2}))) = J$. This implies that $J$ is a finite union of points and analytic arcs. \\
However we can construct a smooth open of $J$ concretely.  For this fix $y_0 \in J$ such that $\partial P(x,y_0)\partial_Y P(x,y_0)\neq 0$ for all $(x,y_0)\in \C^2$ satisfying $P(x,y_0) = 0$. \\
Let $D$ be an open disc in a neighbourhood of $y$ such that there exists $n$, open connected sets $U_1, \dots,U_n$ and analytic functions $h_i$ on $U_i, i = 1, \dots,n $ such that 
\begin{align*}
\{(x,y)\in \C^2; P(x,y) = 0, y \in D \} = \{(x,h_i(x)), x \in U_i\}. 
\end{align*}
 It holds that $h_i(U_i \cap \Psi(S^{\phi^2})) \cap h_j(U_{j} \cap \Psi(S^{\phi^2})) $ is a finite union of points and analytic arcs for $i,j \in \{1, \dots, n\}$. Thus we can find a point  $ y_1 \in D \cap J$ and an open neighbourhood $U$ of $y_1$ such that $U \cap J$ is an analytic arc. The Lemma now follows from Fatou's theorem. 
\end{proof}

Before we proceed we record a Lemma that is a more or less straightforward consequence of Schwarz Lemma. 
\begin{lem}\label{circles} Let $\mathcal{G} $ be an analytic non-zero function in two variables on  $D^2$. Suppose that there exists a sequence of circles $S_n$ whose radius $r_n $ tends to 0 as $n$ tends to infinity such that 
for all $n \in \Z_{\geq 0}$, $\mathcal{G}(x,y) = 0 $ for infinitely many $(x,y) \in S_n^2$. Then there exists a complex number $\theta \in S^1$ such that $\mathcal{G}(x,\theta x) =0$ identically.  
\end{lem}
\begin{proof} Firstly by continuity $\mathcal{G}(0,0) = 0$. We may assume that $\mathcal{G}$ is irreducible \cite{monodromygroup} and then by Puiseux's theorem as given in Lemma \ref{puiseaux} there exists an analytic function $h$ in a neighbourhood $U_0$ of 0 and a positive integer $N$ such that $(x^N, h(x)), x \in U_0$ parametrizes a neighbourhood of 0 of the zero-set of $\mathcal{G}$. For $n$ large enough $S_n$ is contained in $U_0$ and we may renumber $S_n$ such that all of them are.  We have that for each $n \geq 0$  there are  infinitely many $x_n \in \C$ with $ |x_n| = r_n^{1/N}$ and for which 
\begin{align*}
|h(x_n)| = r_n
\end{align*} 
holds. By analytic continuation this then holds for all $x $ of absolute value $r_n^{\frac 1N}$. Thus writing $h$ in polar coordinates $(r,\omega)$ we have $\partial_\omega h(r_0,\omega_0) = 0 $ for all $(r_0,\omega) \in \{r_n, n \in \Z_{\geq 0}\}\times [0,2\pi)$ and we deduce that $\partial_\omega |h(r,\omega)| = 0$ identically. We also have $|h(r,\omega)| = |h(r)| = r^N $ for infinitely $r$ and so $|h(r)| = r^{N}$ identically by analyticity. However this also implies that $h(z) = z^N(1 + O(z))$ and so $h = \tilde{h}^N$ for an analytic function $\tilde{h}(z)$ satisfying $\tilde{h}(0) = 0$. This contradicts irreducibility if $N \geq 2$ and we deduce that $|h(z)| = |z|$. By Schwarz's Lemma    $h(z) = \theta z$ for some $\theta \in S^1$. 
\end{proof}

Now we come to the final proposition of this subsection. 
\begin{prop}\label{proposition} Suppose that $f$ is not exceptional and let $P, \mathcal{B}$ be as in (\ref{arc2}). There exists $\theta \in S^1$ such that 
\begin{align*}
P(\Psi(z), \Psi(\theta z)) = 0
\end{align*}
identically. 
\end{prop}
\begin{proof}  By Lemma \ref{cont1} and Corollary \ref{cor1} and our assumption on $f$ to be non-exceptional
 we may assume that $\overline{C}_0 \cap \left( \partial \mathcal{F}\times \mathcal{F}\cup \mathcal{F}\times \partial \mathcal{F}\right)$ is empty.  By Lemma \ref{cont2} and Lemma \ref{limits} it holds that $ \overline{C}_0 \cap \Psi(S^{\phi^2})^2$ or $ \overline{C_0^h} \cap \Psi^h(S^{\phi^2})^2$ is non-empty. We first assume that the former holds.  As $ \Psi(S^{\phi^2})$ is connected and $\overline{C}_0 \cap \left( \partial \mathcal{F}\times \mathcal{F}\cup \mathcal{F}\times \partial \mathcal{F}\right)$ is empty we can find infinitely many points in $\overline{C}_0 \cap  \Psi(S^{\phi^2})^2$ and thus an analytic arc $\mathcal{B}_1$ in $\overline{C}_0 \cap  \Psi(S^{\phi^2})^2$. If $ \overline{C}^h_0 \cap \Psi^h(S^{\phi^2})^2$ is non-empty we can use the same arguments to prove that this contains an analytic arc. However complex conjugating the arc we see  that then also $\overline{C}_0 \cap  \Psi(S^{\phi^2})^2$ contains an analytic arc.

We can repeat all of the above arguments with $\phi$ replaced by $\phi^2$ to deduce that there is an analytic arc in $\mathcal{C}\cap \Psi(S^{\phi^{4}})\times \Psi(S^{\phi^4})$. We can continue inductively and deduce that $\mathcal{C} \cap \Psi(S^{\phi^{2^N}})\times \Psi(S^{\phi^{2^N}})$  contains a (connected) analytic arc $\mathcal{B}_N$ for all integers $N \geq 0$. Thus setting $\mathcal{G}(x,y) = x^{\deg_X P}y^{\deg_YP}P(\Psi(x), \Psi(y))$ in Lemma \ref{circles} we finish the proof. 
\end{proof}
\subsection{Roations of the basin of infinity}
In this subsection we prove the following Lemma. 
\begin{lem}\label{rotation} Suppose that $f$ is not an exceptional polynomial with connected Julia set and suppose that there exists an irreducible polynomial $P \in \C[X,Y]\setminus \{0\}$ and a complex number $\theta$ of absolute value 1 such that $P(\Psi(z), \Psi(\theta z)) = 0$ for all $z \in D$.  Then $\theta$ is a root of unity. 
\end{lem}

We assume that $\theta$ is not a root of unity and then construct a non-trivial family of symmetries of the Julia set $J(f)$ of $f$ and by a theorem of Levin \cite[Theorem 1]{levinsymmetries} and Douady and Hubbard \cite{DH} the existence of such a family implies that $f$ is exceptional. \\

\begin{proof}[Proof of Lemma \ref{rotation}]
In what follows we will assume that $\theta$ is not a root of unity and derive a contradiction. We first note that the partial derivatives of $P$ do not vanish identically. Otherwise $\Psi$ would be a constant function. We denote by $h_\theta$ the function 
\begin{align*}
h_\theta(z) = \Psi(\theta\Phi(z))
\end{align*}
and by assumption $h_\theta$ is an algebraic function that can be continued to some disc centred at a point of $J(f)$. Thus we can continue $h_\theta$ as an analytic function to an open disc $D$ centred at  $x_0\in J$. Choosing $D$ small enough $h_\theta$ is bi-holomorphic onto its image. Clearly $h_\theta(D) \cap B_\infty$ is open and non-empty and hence we can continue $h_{\theta^{-1}}$ to $h_\theta(D)$ using the relation $h_\theta^{-1}(z) = h_{\theta^{-1}}(z)$ for $z \in h_\theta(D) \cap B_\infty$. It follows that $h_\theta(D \cap \C\setminus B_\infty) \subset  \C\setminus B_\infty$ since $h_{\theta^{-1}}(B_\infty) = B_\infty$. As also $h_{\theta}(B_\infty) = B_\infty$ we have $h_\theta(D\cap J) = h_\theta(D)\cap J$. Thus $h_\theta$ is a symmetry of $J$. \\
As the repelling periodic points of $f$ are dense in $J$ we may assume that the centre of $D$ is a repelling fixed point of $f$ (after possibly replacing $f$ by an iterate of $f$). Thus there exists a holomorphic function $F$ such that 
\begin{align*}
f(F(z)) = F(\lambda z), F(0) = x_0, F'(0) = 1, |\lambda| > 1.
\end{align*}
There exists a disc $D_0$ about 0 such that $F(D_0) \subset D$ and $F$ is bi-holomorphic onto its image.  Setting $U = F(D_0)$ we see that a branch of the inverse of $f$ defined by  $f^{-1}(z) = F(\lambda^{-1}F^{-1}(z))$ is well-defined on $U$ and it satisfies  $f^{-1}(U) \subset U$. Thus we can define $h_{\theta^d}(z)$ on $U$ by 
\begin{align*}
h_{\theta^d}(z) = f\circ h_\theta \circ f^{-1}(z).
\end{align*}
As $h_\theta (J\cap U) = h_\theta(U) \cap J$ and $J$ is completely invariant under $f$ we have that $h_{\theta^d}(U\cap J) = h_{\theta^d}(U)\cap J$. Thus the inductively constructed family  $H= \{g_n\}_{n \geq 0}$ with $g_{n+1} = f\circ g_n\circ f^{-1}(z), g_0 = h_\theta$ is a family of symmetries of $J$. Now we prove that $H$ is a normal family. \\

We have that $U \cap B_\infty \subset V = \{g_f(z) \leq t\}$ for some fixed $t > 0$ and thus that $g_n(U\cap B_\infty) \subset V$ for all $n \geq 0$. We also proved above that $g_n(\C\setminus B_\infty) \subset \C\setminus B_\infty$ and thus $g_n(U) \subset V$ for all $n \geq 0$ and thus $g_n$ is uniformly bounded on $U$.  By Montel's theorem the family $H$ is a normal family. Moreover the set of rotations of the unit disc is a closed subset of the space of functions on the unit disc with respect to the sup-norm.   Thus $g_n$ is a non-trival normal family of symmetries of $J$ in the sense of Levin.  From his  theorem \cite{levinsymmetries} and \cite{DH} follows that $f$ is exceptional.  
\end{proof}
The next lemma is an addendum to the previous one but it uses the precise classification of periodic varieties in the work of Medvedev and Scanlon \cite{medvedevscanlon}.  For the following Lemma we do not assume that the Julia set of $f$ is connected. 
\begin{lem}\label{medvedev}  Suppose that $f$ is  a non-exceptional polynomial of degree $d \geq 2$ and suppose that there exists an irreducible polynomial $P \in \C[X,Y]\setminus \{0\}$ and a root of unity $\theta$ such that $P(\Psi(z), \Psi(\theta z)) = 0$ for all $z \in D_R$. Then there exists a linear polynomial $L$ commuting with an iterate $f^{\circ M}$ of $f$ such that $P$ divides $f^{\circ M}(X) - L\circ f^{\circ M}(Y)$.  
\end{lem}
\begin{proof} We first note that if $\theta \in \mu_{d^{\infty}}$ we can set $L(X) = X$ and we are finished. We now assume that $\theta \notin \mu_{d^{\infty}}$. As $\Psi(z), \Psi(\theta z)$ are algebraically dependent so are $\Psi(z^{d^n}) = f^{\circ n}(\Psi(z)), \Psi(\theta^{d^n} z^{d^n}) = f^{\circ n}(\Psi(\theta z))$ for every $n \in \Z_{\geq 0}$. Thus picking $n$ large enough we may assume that the order of $ \tilde{\theta} = \theta^{d^n}$ is co-prime to $d$ (but larger than 1) and there exists an irreducible  non-zero polynomial $\tilde{P}$ such that   $\tilde{P}(f^{\circ n}(\Psi(z)), f^{\circ n}(\Psi(\theta z)) = 0$. As $\tilde{\theta}$ has order co-prime to $d$ there exists $N \geq 1$ such that $\tilde{\theta}^{d^N} = \tilde{\theta}$. We deduce that the curve defined by $\tilde{P}$ is invariant under the coordinate action of $f^{\circ N}$. By \cite[Theorem 6.24]{medvedevscanlon} there exists $\ell \in \Z_{\geq 0}$, a polynomial $h$ of positive degree and a linear polynomial $L$ commuting with $f^{\circ N}$ such that up to a constant $\tilde{P} = Y -L\circ h^{\circ \ell}(X)$ or $\tilde{P} = X -L\circ h^{\circ \ell}(Y)$. We may assume the former and if $\ell \geq 1$ it follows that the function 
\begin{align*}
L\circ h^{\circ \ell}(X) = \Psi(\tilde{\theta} \Phi(X))
\end{align*}
has meromorphic continuation to $\infty$ where it has a superattracting fixed point of degree at least 2. Since  $\Psi$ is 1-1 at infinity we get a contradiction and so $\ell = 0$ and we obtain that $\tilde{P} = y - L(X)$ for $L$ a linear polynomial satisfying $L\circ f^{\circ N} = f^{\circ N}\circ L$. Setting $M = N +n$ we conclude the proof. 
\end{proof}
\subsection{Proof of Theorem \ref{equipotential} for connected Julia sets}
We recall that we may assume that $\mathcal{C}$ is defined by a polynomial $P$ that satisfies $\partial_XP\partial_YP \neq 0$. If $\mathcal{C}$ intersects $L_r^2$ in infinitely many points we can find $\mathcal{B}$ as in (\ref{arc2}). From Proposition \ref{proposition} and Lemma \ref{rotation} we deduce that $P(\Psi(z), \Psi(\theta z)) = 0$ identically on $D$ for some $\theta \in \mu_{\infty}$. Finally  Lemma \ref{medvedev} finishes the proof. 

\section{Disconnected Julia sets}\label{disconnectedjulia}
In this section we consider as in the last section a monic complex polynomial $f\in \C[X]$ of degree $d \geq 2$. However in this section we suppose that  its Julia set is  disconnected. (In particular $f$ is non-exceptional.) \\

We pause quickly to discuss the proof strategy. We use the same strategy as in the last section but instead of using the rigidity theorems of Fatou and Levin we rely on investigations of the mondromy action on $\Psi$ as discussed below. In this sense our proof is independent of previous work on the geometry of Julia sets and introduces a new perspective. We can make the proof also independent of Fatou's theorem although we do not explicitly point it out below.  This shows that monodromy contains ``enough information" about the geometric structure of the dynamical system.   \\

If the Julia set of $f$ is disconnected then the basin of infinity $B_\infty$ contains a critical point of $f$. Thus it is useful to define
\begin{align*}
C_{\text{crit}} = \{c \in B_\infty; f'(c) = 0\}. 
\end{align*}
and the inverse images
\begin{align*}
C_n = f^{-\circ n}(C_{\text{crit}}), C_\infty = \cup_{n \geq 0}C_n.
\end{align*}
We can compute $R$ using $C_{\text{crit}}$. Namely
\begin{align*}
R = \min \{\exp(-g_f(c)); c \in C_{\text{crit}} \} 
\end{align*}
and it is immediate that we can continue $\Psi$ analytically to 
$$D_{R} = \{z \in \mathbb{C}; |z| < R\}$$ 
using the relation
\begin{align*}
\Psi(z^d) = f(\Psi(z))
\end{align*}
but not to a larger disc. But we can  continue $\Psi$ along certain paths beyond $D_R$ but not as a global analytic function. In order to describe the situation more precisely we introduce some terminology. We denote by $D_R^* = D_R\setminus \{0\}$ and by $D_n^* = D^*_{R^{1/d^n}}$. We define an auxilliary set 
\begin{align*}
E_n = \{z \in \C; \Psi(z^{d^n}) = f^{\circ n_c}(c) \text{ for some } c \in C_{\text{crit}}\}
\end{align*}
where $n_c \in \mathbb{Z}_{\geq 1}$ is minimal among the integers $n$ such that $f^{\circ n}(c) \in \Psi(D_R^*)$. For a non-negative integer $n$ we define the set $S_n$ to be the local analytic functions $\tilde{\Psi}$ on $D_n = D^*_n \setminus U_n$, where 
$U_n = \{\zeta z; z\in \cup_{m = 1}^nE_m, \zeta \in \mu_{d^n}\}$,  that satisfy 
\begin{align}\label{n}
\Psi(z^{d^n}) = f^{\circ n}(\widetilde{\Psi}(z)).
\end{align}
Note that the left hand side of (\ref{n}) is a well-defined analytic function. We can interpret $\widetilde{\Psi}$ as the algebraic function $f^{-\circ n}(Y)$ in the variable $Y = \Psi(z^{d^n})$. Note that we can restrict the solutions to $D^*_0 = D^*_R$ where they are analytic and we see that a complete set of solutions of (\ref{n}) is given by the set of functions $\{\Psi(\zeta z); \zeta \in \mu_{d^n}\}$. \\

From the interpretation of  $\tilde{\Psi}$ as an algebraic function in $\Psi(z)$ we see that each function element on $D_n$ is given by a triple $(\zeta, D, \gamma) $ where $D$ is an open disc in $D_n$, $\gamma$ a path from $D_0$ to $D$ and the function $\tilde{\Psi}$ on $D$ is the continuation of $\Psi(\zeta z)$ from $_0$ to $D$ along $\gamma$. Moreover at each point $z_0 \in E_n$ restricting a function element $\widetilde{\Psi}$ to a disc $D$ centred at $z_0$ we obtain a Puiseux expansion of $\widetilde{\Psi}$ converging on  $D$ and a well defined value $\widetilde{\Psi}(z_0)$. As each function element $\widetilde{\Psi}$ is a continuation of $\Psi(\zeta z)$ for some $\zeta \in \mu_{d^n}$ it follows that continuing $\widetilde{\Psi}$ along a simple closed loop around $z_0$ sends $\widetilde{\Psi}(z)$ to $\widetilde{\Psi}(\zeta z)$ for some $\zeta \in \mu_{d^n}$ where $\widetilde{\Psi}(\zeta z)$ is given by the triple $(\zeta \zeta_1, D, \gamma)$ when $\widetilde{\Psi}$ is given by $(\zeta_1, D, \gamma)$. Now if $\widetilde{\Psi}(z_0) \in C_n$ then the monodromy action given by loops in $D$ around $z_0$ we just described is non-trivial. We can say something a bit more precise for continuations of $\Psi$ to $D_n$. For this we define the set $S_{n,\Psi}$ to be the set of locally analytic functions on $D_n$ that are continuations of $\Psi$. Clearly $S_{n, \Psi} \subset S_n$ and we will write $\Psi$ for a continuation of $\Psi$ to $D_n$. We have the following  Lemma. 
\begin{lem}\label{monodromy} Let $\Psi$ be a continuation of $\Psi$ (on $D_0)$ along a path in $D_n$ to a small disc $\tilde{D}$ centred at $z_0 \in E_n$ (as an algebraic function of $\Psi(z^{d^n})$) and suppose that $\Psi(z_0) \in C_n$. Continuing $\Psi$ restricted to a disc $D \subset \tilde{D}\setminus \{z_0\}$  along a small loop around $z_0$ sends $\Psi(z)$ to $\Psi(\zeta z)$ for $\zeta \in \mu_{d^{\infty}}$ of order at least $2^n$.   
\end{lem}
\begin{proof} We prove this by induction on $n$. First consider 
\begin{align*}
\Psi(z^d) = f(\Psi(z^d))
\end{align*}
on $D_1$ and recall that $S_{1,\Psi} \subset S_1$. Now if $\Psi(z_0) \in C_1$ then $f'(\Psi(z_0)) = 0$ and thus the monodromy action on $\Psi$ is non-trivial. It follows that continuing $\Psi$ along a simple closed loop around $z_0$ sends $\Psi(z)$ to $\Psi(\zeta z)$ where $\zeta \in \mu_d$ is of order at least 2. \\
Now consider 
\begin{align*}
\Psi(z^{d^n}) = f^{\circ n}(\Psi(z))
\end{align*} 
on $D_n$ for $n \geq 2$ and $z_0 \in E_n$ such that $\Psi(z_0) \in C_n$ (for the continuation that we consider). Making the substitution $w = z^d$ it follows from $\Psi(z^d) = f(\Psi(z))$ that $\Psi(z^{d^n}) = f^{\circ n-1}(\Psi(z^d))$. Noting that $z_0^{d} \in E_{n-1}, \Psi(z_0^d) \in C_{n-1}$ we deduce that continuing  $\Psi(w)$ around $z_0$ sends $\Psi(w)$ to $\zeta_1 w$ with $\zeta_1 \in \mu_{d^{n-1}} $ of order at least $2^{n-1}$. Moreover continuing $\Psi$ around $z_0$ sends $\Psi(z)$  to $\Psi(\zeta_2 z)$ with $\zeta_2 \in \mu_{d^n}$. From $\Psi(z^d)= f(\Psi(z))$ follows that $\Psi(\zeta_1 z^d) = \Psi(\zeta_2^dz^d)$ and thus $\zeta_2^d = \zeta_1$ and $\zeta_2$ has order at least $2^n$. 
\end{proof}
\subsection{Puiseux expansions}

As we will be using Puiseux's theorem in a lot of the proofs we give the version we are going to use as a Lemma. 
\begin{lem}\label{puiseaux} (Puiseux) Let $\mathcal{G}$ be an analytic function on $D^2$ that satisfies $\mathcal{G}(0,0) = 0$.  There exists an integer $N \geq 1$ and $\epsilon > 0$ such that for any integer $M$ divisible by $N$ and any analytic function $w:D \rightarrow D_\epsilon $ where $D$ is a disc about a point $x_0 \in \C$, that satisfies $w(x_0) = 0$ and whose first $M-1$ derivatives vanish at $x_0$ there exists an analytic function $h_w$ on $D$ such that $\mathcal G(w(x), h_w(x))= 0$ for all $x \in D$.  
\end{lem}
\begin{proof} This follows from \cite[The Puiseax Theorem p.28]{monodromygroup}. We only need to note that $w$ sends simple loops about $x_0$ to loops winding $M$ times about 0. 
\end{proof}
We will describe the Puiseux expansions of $\Psi$ a bit more concretely in what follows. 
\begin{lem} \label{cont} At each point $z_0 \in D_n^*$ and each $w_0 \in \{f^{-\circ n}(\Psi(z_0^{d^n}))\}$ there exists a disc $D$ about $z_0$, a positive  integer $N$, and an analytic function $\Psi_N$ on $D$ such that $S= \{z \in D; (z_0 +(z-z_0)^N, \Psi_N(z))\}$  is the set of $(z,w) \in \C^2$ such that
\begin{align*}
\Psi(z^{d^n}) - f^{\circ n}(w)= 0, (z_0,w_0) \in S. 
\end{align*}

\end{lem} 
\begin{proof} 
Setting $\mathcal{G} = \Psi(z^{d^n}) - f^{\circ n}(w)$  in Lemma \ref{puiseaux} we obtain the claim after noting that the zero-set of $\mathcal{G}$ is smooth and irreducible in a neighbourhood of each point $(z_0,w_0)\in (D_n^*)^2$. 

\end{proof}
Similarly as in the previous section we define
\begin{align*}
\mathcal{F} = \bigcup_{n \geq 0}f^{-\circ n}(\Psi(A(\phi^2, \phi)).
\end{align*}
Recall that $\phi \in D_R\cap \R$. For any $a < b$ we define $A(a,b) = \{z \in \C; a<|z|<b\}$. There exists $\epsilon > 0$ such that the function $\mu$ defined by 
\begin{align*}
\mu(X) = \Psi^h(\phi^2/\Phi(X))
\end{align*}
 is well-defined on $\Psi(A(\phi-\epsilon, \phi + \epsilon))$. We start this section by investigating various properties of this function. From the discussion of the last section already follows that it can not be defined globally on $B_\infty$ but we can make sense of it as a locally defined function on paths. 
We have the following Lemma. 
\begin{lem}\label{contmu} Let $\gamma:[0,1)\rightarrow \mathcal{F}$ be a path with $\gamma(0) \in \Psi(S^\phi)$. We can continue $\mu$ along this path and $\mu_{|\gamma|}$ is continuous and has a Puiseax expanison at each point of $|\gamma|$. Moreover we can continue $\mu^h$ along $\mu\circ \gamma$ such  that it holds that $\mu^h\circ \mu(z) = z$ on $|\gamma|$. 
\end{lem}
\begin{proof}  It is clear that the claim is true for $\gamma$ restricted to $[0,T)$ for some $T>0$. We pick $T$ maximal such that the claim is true and are going to show that $T = 1$. We prove this by contradiction. Suppose that $T <1$. Then $\gamma(T) = x_0 \in \mathcal{F}$ and $\mu$ has a Puiseux expansion at $x_0$. Let $D$ be the disc at which the function $\Psi_N$ in Lemma \ref{cont} is analytic. Moreover $\mu^h$ has a Puiseux expansion at $\mu(x_0)$ where the value $\mu(x_0)$ is well-defined by the definition of the Puiseux expansion. We can shrink $D$ such that $\mu(D)$ is contained in a disc on which the Puiseux expansion of $\mu^h$ at $\mu(x_0)$ converges. Let $T' < T$ be such that $\gamma(T') \in \partial D$ and such that $\gamma(t) \notin\partial D$ for any $T'<t<T$. Then we can continue $\mu$ along any path that consists of $\gamma$ restricted to $[0,T']$  composed with any path in $D\setminus \{x_0\}$.  Along any such path it holds that $\mu^h\circ \mu(z) = z$. For $D$ sufficiently small $D \cap |\gamma|$ consists of two connected open components $D_1,D_2$ and $\mu$ is analytic on $D_1$ and $\mu^h$ is analytic on $\mu(D_1)$. On $D_1$ it holds that $\mu^h\circ \mu(z) = z$ and then by continuity also on $|\gamma|$ if we continue $\mu$ to $|\gamma|$ from $D_1$. Thus we can continue $\mu$ with the properties demanded of the Lemma to $\gamma$ restricted to $[0,T^{''})$ with $T'' > T$. A contradiction. 
\end{proof}

\begin{lem} \label{continuation3} Let  $P_1,P_2 \in \C[X,Y]$ and suppose that there exists a non-constant analytic function $h$ in a neighbourhood of $x_0 \in \Psi(S^{\phi})$ such that     
\begin{align}\label{rel1}
P_1(x,h(x)) = P_2(\mu(x), \mu(h(x))) = 0
\end{align}
 and $h(x_0) = y_0 \in \Psi(S^{\phi})$. Suppose there is a path $\gamma: [0,1) \rightarrow \mathcal{F} \times \mathcal{F}$ with $\gamma(0) = (x_0,y_0) \in \Psi(S^{\phi})^2$ and such that $P_1\circ \gamma = 0$. Then for any continuation of $(\mu, \mu)$ along $ \gamma $  it holds that $P_2(\mu(x),\mu(y)) = 0$ for all $(x,y) \in |\gamma|$. 
\end{lem}
\begin{proof} 
Let $T$ be maximal such that the claim is true for $\gamma$ restricted to $[0,T)$. By assumption $T > 0$ and we are going to show that $T = 1$. If $T < 1$ then $\gamma(T) \in \mathcal{F}$ and we can set $x_0 = \gamma(T)$ for $x_0$ as in Lemma \ref{contmu}. By Lemma \ref{puiseaux} and Lemma \ref{contmu} there exists a positive integer $N$, analytic functions $h_{N}, \mu_{N_1}$ in a neighbourhood of $x_0$ and an analytic function $\mu_{N_2}$ in a neighbourhood of $h(x_0)$ as well as $\epsilon > 0$ such that 
\begin{align*}
\mu_{N_1}(y) & = \mu(x_0 + (y-x_0)^{N}),\\ h_{N}(y)  & = h(x_0 + (y-x_0)^{N}), \\
\mu_{N_2}(h_N(y)) & = \mu(h(y_0 + (x-y_0)^N))
\end{align*}    
for all $y $ satisfying $x_0 + (y-y_0)^N \in \gamma(T-\epsilon, T)$. Thus we can continue (\ref{rel1}) by     
\begin{align*}
P_1(y_0 + (y-y_0)^{N},h_N(y)) = P_2(\tilde{\mu}_N(y), \tilde{\mu}_N(h_N(y))) = 0
\end{align*}
to a neighbourhood of $x_0$. This is a contradiction. 
\end{proof}
Now we record an analogue of Lemma \ref{limits}. 
\begin{lem} \label{limits2} Let $\gamma:[0,1] \rightarrow \overline{\mathcal{F}}$ be a path such that $\gamma([0,1)) \subset \mathcal{F}$ and $\gamma(0) \in \Psi(S^\phi), \gamma(1) \in \Psi(S^{\phi^2})$. Then for any limit point $c$ of $\mu\circ \gamma(t)$ as  $t \rightarrow 1$ it holds that $c \in J^h$. Moreover for any path $\gamma:[0,1] \rightarrow \overline{\mathcal{F}}$ with $\gamma([0,1))\subset \mathcal{F} $ and $\gamma(0) \in \Psi(S^\phi), \gamma(1) \in J$ we have that for any limit point $c$ of $\mu\circ\gamma$ it holds that $c   \in \Psi^h(S^{\phi^2})$.
\end{lem} 
\begin{proof} This follows from Lemma \ref{contmu}.
\end{proof}

\subsection{Analytic continuation.}
We assume that for $\exp(-r) = \phi$, $\mathcal{C}\cap L_r^2$ contains an analytic arc $\mathcal{A}$. We let $C_0$ be the connected component of $\mathcal{C}\cap \mathcal{F}^2$ containing $\mathcal{A}$. We recall that relation (\ref{arc2}) holds on $\mathcal{A}$. 
Now we prove an analogue of Lemma \ref{cont1} for disconnected Julia sets. 
\begin{lem} \label{empty2} Suppose that $\overline{C_0}\cap \partial \mathcal{F}\times \mathcal{F}\cup \mathcal{F}\times \partial \mathcal{F}$ is not empty. Then there exists $x_0 \in \Psi(S^{\phi^2})$ such that we can continue $\mu$ analytically along a path $\gamma:[0,1] \rightarrow \overline{\mathcal{F}}$ with $\gamma([0,1)) \subset \mathcal{F}, \gamma(0) \in \Psi(S^\phi), \gamma(1) = x_0$.  
\end{lem}
\begin{proof} We frist note that we can continue $\mu$ to a point on $\Psi(S^{\phi^2})$ if and only if we can continue $\mu^h$ to a point on $\Psi^h(S^{\phi^2})$. \\From the relation (\ref{arc2}) and Lemma \ref{continuation3} follows that there exists a path $\tilde{\gamma}:[0,1]\rightarrow \overline{\mathcal{F}}^2, \tilde{\gamma}(0) \in \Psi(S^{\phi})^2, \gamma(1) \in \partial \mathcal{F}\times \partial \mathcal{F}$ such that 
\begin{align*}
P(x,y) = P(\mu\circ x, \mu\circ y) = 0 
\end{align*}
for all $(x,y) \in |\tilde{\gamma}|$ and by Lemma \ref{contmu} it follows that $P(\mu^h\circ x, \mu^h \circ y) = P^h(x,y) = 0$ for all $(x,y) \in \mu\circ \tilde{\gamma}$. (Here we have taken the appropriate continuations of $\mu^h$ as described in Lemma \ref{contmu}). Thus up to replacing $P$ by $P^h$ and permuting $(x,y)$ we may assume that there exists $(x_0,y_0) \in \overline{C}_0\cap \Psi(S^{\phi^2})\times \mathcal{F}$ and as $\Psi(S^{\phi^2})$ does not contain isolated points and $\mathcal{F}$ is open we may even assume that $\partial_YP(x_0,y_0)\partial_X P(x,\mu(y_0)) \neq 0$ for all $x$ such that $P(x,\mu(y_0)) = 0$.  Thus there exists an analytic function $g$ on a disc $D$ centred at $x_0$ such that 
\begin{align*}
P(x,g(x)) = P^h(\mu(x), \mu\circ g(x)) = 0
\end{align*} 
for all $x \in D_1 \subset \mathcal{F}$, where $D_1$ is such that $D\setminus \Psi(S^{\phi^2}) = D_1\cup D_2$ with $D_2 \subset B_\infty \setminus \mathcal{F}$. Let $\mu_0$ be some limit point of  $\mu(x)$ for $x \rightarrow x_0, x \in D_1$. There exists an analytic function $\tilde{g}$ on a disc $\tilde{D} $ centred at $\mu\circ g(x_0)$ such that $\tilde{g}(\mu\circ g(x_0)) = \mu_0$ and $P^h(\tilde{g}(z),z ) = 0$ for all $z \in \tilde{D}$. We can shrink $D$ sufficiently such that $\mu\circ g(D)\subset \tilde{D}$ and we see that  $\tilde{g}\circ \mu\circ g$ is a continuation of $\mu$ to $D$. 
\end{proof}
We get as a direct consequence of the previous Lemma the analogue of Corollary \ref{cor1}. 
\begin{corollary}\label{cor2} The set $\overline{C_0}\cap \left(\partial \mathcal{F}\times \mathcal{F}\cup \mathcal{F}\times \partial \mathcal{F}\right)$ is  empty. 
\end{corollary}
\begin{proof}  
Suppose for contradiction that $\overline{C}_0\cap \left(\partial \mathcal{F}\times \mathcal{F}\cup \mathcal{F}\times \partial \mathcal{F}\right)$ is not empty. 
By Lemma \ref{empty2} we can continue $\mu$ analytically to a disc $D$ centred at a point of $\Psi(S^{\phi^2})$.  \\
We first note that it is sufficient to show that there exists $x_0 \in D$ such that $\mu(x_0) \in C_{\infty}$. For if that is true then there exists $n$ such that $f^{\circ n'}(x_0) = 0$ and  as $f^{\circ n}(\Psi(\phi^2/\Phi(X))) = \Psi(\phi^{2d}/\Phi^d(X))$ we obtain that the equation $f^{\circ n}(z) = w$ as an analytic solution $z = g(w)$ satisfying $g(f^{\circ n}(x_0)) = x_0$. This contradicts $f^{\circ n'}(x_0) = 0$. \\

Now $\Phi(D)$ contains an open disc $D'$ centred at a point of $S^{\phi^2}$ and $\phi^2/\Phi(D')$ contains an open set $U$ of the form $ U = \{r\exp(2\pi i \epsilon); r_1 < r<1, \epsilon \in I \}$ for some open interval $I \subset (0,1)$. It is clear that for $n$ large enough the set 
$\{z^{d^n}; z \in U\}$ contains an annulus $A_0$ of the form $A_0 = \{z \in \C; r_1' <|z|< r_2'\}$ for some $r_1'<R< r_2'$ and then  we can deduce that there exists $z_0 \in U$ such that $f^{\circ n}(\Psi(z_0)) = \Psi(z_0^{d^n}) = c_0 $ with $f'(c_0) = 0$. This implies that $ \Psi(z_0) \in C_{\infty}$ and we are finished.   
\end{proof}
Now we need an analogue of Lemma \ref{cont2}. 
\begin{lem}\label{noidea} Suppose that $\overline{C}_0 \cap \left( \partial \mathcal{F}\times \mathcal{F}\cup \mathcal{F}\times \partial \mathcal{F}\right)$ is empty. Then $\overline{C}_0 \cap (\Psi(S^{\phi^2})\times \Psi(S^{\phi^2}) \cup J^2) $ is non-empty. 
\end{lem}
\begin{proof} Suppose for contradiction that $\overline{C}_0 \cap (\Psi(S^{\phi^2})\times \Psi(S^{\phi^2}) \cup J^2) $ is empty. As  in the proof of Lemma \ref{cont2} we deduce from Lemma \ref{continuation} that $J = \pi_2(\pi_1^{-1}(\Psi(S^{\phi^2}))\cap \overline{C}_0)$ and in particular that $J$ has finitely many connected components. This contradicts the fact that $J$ has infinitely (in fact uncountably) many connected components. 
\end{proof}
Now we prove an analogue of proposition \ref{proposition} and the proof is in fact quite similar. 
\begin{prop} \label{proposition2} Let $f$ be, as above, a polynomial with disconnected Julia set and $P$ be such that there exists $\mathcal{B}$ as in (\ref{arc2}). Then there exists $\theta \in S^1$ such that $P(\Psi(z), \Psi(\theta z)) = 0$ for $z \in D_R$. 
\end{prop}
\begin{proof}
We first recall that by Corollary \ref{cor2}  $\overline{C}_0 \cap\left(\partial \mathcal{F} \times \mathcal{F}\cup \mathcal{F}\times \partial \mathcal{F}\right)$ is empty and thus also $\overline{C}^h_0 \cap\left(\partial \mathcal{F}^h \times \mathcal{F}^h\cup \mathcal{F}^h\times \partial \mathcal{F}^h\right)$ is empty. From Lemma \ref{noidea} follows that $\overline{C}_0 \cap (\Psi(S^{\phi^2})\times \Psi(S^{\phi^2}) \cup J^2)$ is non-empty. We now prove that $\mathcal{C} \cap \Psi(S^{\phi^2})^2$ contains an analytic arc. From Lemma \ref{limits} we deduce that $\overline{C}_0 \cap\Psi(S^{\phi^2})^2 $ or $\overline{C}^h_0\cap\Psi^h(S^{\phi^2})^2 $ is non-empty. However those two conditions are equivalent which can be seen by applying complex conjugation  and as $\Psi(S^{\phi^2})$ is connected we deduce that $\mathcal{C} \cap\Psi(S^{\phi^2})^2 $ contains infinitely many points and thus an analytic arc. \\

We may finish the proof as in the proof of Proposition \ref{proposition}. 
\end{proof}

Now we prove the analogue of Lemma \ref{rotation} for $f$ with disconnected Julia set. 
\begin{lem}\label{rotation2} Let $f$ be a polynomial with disconnected Julia set and suppose there exists $\theta \in S^1$ and an irreducible polynomial $P \in \C[X,Y]\setminus \{0\}$such that $P(\Psi(z), \Psi(\theta z)) = 0$ for $z \in D_R$. Then $\theta \in \mu_\infty$. 
\end{lem}
\begin{proof} We first  prove that we may assume that there exists $N \in \Z_{\geq 0}$ such that for all $n \geq N$ it holds that if  $z \in E_n$ then $\theta z \notin E_n$.  Suppose to the contrary that there exist infinitely many $n \in \Z_{\geq 1}$ such that there exists $z \in E_n$ with $\theta z \in E_n$. It follows that $z^{d^n}, (\theta z)^{d^n} \in \{\Phi(f^{\circ {n_c}}(c)); c \in C_{\text{crit}}\}$ and so $\theta^{d^n}$ is contained in a finite set independent  of $n$ for infinitely many $n$. It follows that $\theta \in \mu_{\infty}$.  

We assume now that  $N$ as described above exists. We can then pick $n$ such that $2^n > \deg(P)$ and continue the relation $P(\Psi(z), \Psi(\theta z)) = 0$ to some $z_0 \in E_n$ such that for the continuation of $\Psi(z)$ holds $\Psi(z_0) \in C_n$. By Lemma \ref{monodromy} continuing $\Psi(z)$ along a small simple loop around $z_0$ sends $\Psi(z)$ to $\Psi(\zeta z)$ with $\zeta \in \mu_{d^n}$ of order at least $2^n$. Thus we obtain $P(\Psi(\zeta z), \Psi(\theta z)) = 0$. Now we can continue this relation to $\zeta^{-1}z_0$ and continue it along a small loop around $\zeta^{-1}z_0$. As $\zeta^{-1}z_0 \in E_n$, $\theta \zeta^{-1}z_0 \notin E_n$ and thus we obtain $P(\Psi(\zeta^2 z), \Psi(\theta z))= 0$. We can continue this procedure with $\zeta$ replaced by $\zeta^2$ and then continue inductively to obtain $P(\Psi(\zeta^i z), \Psi(z)) = 0, i = 0,1,\dots, 2^{n}-1$. As $2^n > \deg(\mathcal{C})$ we deduce that $\partial_X P = 0$ which contradicts the fact that $\Psi$ is non-constant.

\end{proof}

\subsection{Proof of Theorem \ref{equipotential} for polynomials with disconnected Julia sets}
\begin{proof} We deduce Theorem \ref{equipotential} from Proposition \ref{proposition2} and Lemma \ref{rotation} as well as Lemma \ref{medvedev}. 
\end{proof}
\subsection{Proof of Theorem \ref{transc}}
If $L_r$ contains a semi-algebraic curve then so does $f^{\circ n}(L_r)$ and thus we may assume that $L_r \subset D_R$ and that $L_r = \Psi(S^{\phi})$. Applying complex conjugation and arguing similarly as in the deduction of the relation (\ref{arc2}) we deduce that there exists a polynomial $P \in \C[X,Y]\setminus \{0\}$ such that $P(X,\mu(X)) = 0$ identically. Thus we can continue $\mu$ analytically to a point on $\Psi(S^{\phi^2})$ and deduce from Corollary \ref{cor1} or the proof of Corollary \ref{cor2} that $f$ is exceptional.    
\section{Proof of Theorem \ref{thm}}\label{proofofthm1}
 
We already noted that there exists a place  $v$ of $K$ such that $|f^{\circ n}(\alpha)|_v\rightarrow \infty$ as $n\rightarrow \infty$. We also remarked that the theorem for $\alpha$ follows from the theorem for $f^{\circ n}(\alpha)$ thus we may assume that $\alpha \in \Psi(D_R)$ where $R$ is the convergence radius of $z\Psi(z)$ at the place $v$. If $v$ is finite then Theorem \ref{thm} follows directly from Proposition \ref{prop1}. If however the place $v$ is infinite then assuming that the curve $\mathcal{C}$ has infinite intersection with $\mathcal{S}_\alpha^2$, it also has infinite intersection with $L_r^2$ for some $r>0$ such that $L_r\subset \Psi(D_R)$. From Theorem \ref{equipotential} follows that $\mathcal{C}$ is pre-periodic by $(f,f)$ and Theorem \ref{thm} follows from Lemma \ref{reduction}.

\section{Proof of Theorem \ref{grandorbit}}\label{proofofthm2}
In this section we assume that $\mathcal{C}$ is not pre-periodic and that it is not fibral and will show that $\mathcal{C} \cap \mathcal{G}_\alpha$ is finite. With Lemma \ref{reductionprep} that suffices to conclude Theorem \ref{grandorbit}. Suppose that $\alpha \in \Psi(D_{v}^*)$ for some finite place $v$ of $K$ of good reduction that is co-prime to $d$.  
For an element $(\beta_1,\beta_2) \in \mathcal{G}_\alpha^2 \cap \mathcal{C}(\overline{\Q})$ we set $n_1, m_1, n_2,m_2$ to be the minimal positive integers such that 
\begin{align}
f^{\circ n_1}(\beta_1) & = f^{\circ m_1}(\alpha),\label{rel1}\\
f^{\circ n_2}(\beta_2) & = f^{\circ m_2}(\alpha) \label{rel2}.
\end{align}
The proof goes by bounding $n_1,n_2,m_1,m_2$ in several steps. In what follows we set 
\begin{align*}
\ell_1 = m_1 - n_1 , \ell_2 = m_2 - n_2.
\end{align*} 
We start by collecting some facts about Kummer extensions. 
\subsection{Kummer extensions}
We first note that as $|d|_v = 1$ the extension $K_v(\mu_{d^\infty})/K_v$ is an unramified extension \cite[Satz 7.13]{neukirch}. In particular the value group of $K_v(\mu_{d^\infty})$ is the same as the value group of $K_v$. It is straightforward to check (for example by looking at ramification) that for $a \in K_v$ it holds that 
 
\begin{align}\label{kummer}
[K_v(a^{1/d^n}, \mu_{d^\infty}): K_v(\mu_{d^\infty})] \geq |(K_v(a^{1/d^n}, \mu_{d^\infty}))^*/(K_v(\mu_{d^\infty}))^*|
\end{align}
where the right hand side is the cardinality of the index as multiplicative groups.  We get the following immediately.  
\begin{lem} Let $K, \alpha, \phi$ be as in the previous sections and $d$ as above. There exists a constant $C_{\text{Kum}} > 0$ depending only on $\alpha, K$ such that 
\begin{align*}
[K_v(\mu_{d^\infty}, \Phi(\alpha)^{1/M}): K_v(\mu_{d^\infty})] \geq C_{\text{Kum}} M
\end{align*}
for any integer $M$ dividing a power of $d$. 
\end{lem}
\begin{proof}  As $\Phi(\alpha) \in K_v$ and the value group of $K_v(\mu_{d^\infty})$ is the same as the one of $K_v$  the Lemma follows from (\ref{kummer}).  
\end{proof}

As a consequence of these observations we obtain the following Lemma. 
\begin{lem}\label{min} There exists a constant  $C_{\text{min}} > 0$ such that $|\min\{\ell_1,0\} - \min\{\ell_2,0\} | \leq C_{\text{min}} $. 
\end{lem}
\begin{proof} We may assume that $\ell_1 < \ell_2$  and note that from the above follows that there is a constant $C_{P} > 0$ such that $[K_v(\beta_1, \mu_{d^\infty}):K_v(\beta_2, \mu_{d^\infty})] \geq C_Pd^{|\min\{\ell_1,0\} - \min\{\ell_2,0\}|}$ and as $\mathcal{C}$ is not special the claim follows. 
\end{proof}

\subsection{Heights}
For the canonical height $\hat{h}_f$ associated to $f$ holds that 
\begin{align} \label{1}
\hat{h}_f(\beta_1) & = d^{\ell_1}\hat{h}_f(\alpha)\\
\hat{h}_f(\beta_2) &  = d^{\ell_2}\hat{h}_f(\alpha).\label{2}
\end{align}

We denote by $d_1,d_2$ the degree of the coordinate functions of $\mathcal{C}$ respectively. 
By a theorem of Néron the following height inequality holds between the coordinates of a curve. 
\begin{align} \label{height}
|d_1h(x) - d_2h(y)| \leq c\sqrt{1 + \min\{h(x),h(y)\}}
\end{align} 
where $(x,y) \in \mathcal{C}(\overline{\Q})$ and $c = c(\mathcal{C})$ depends only on $\mathcal{C}$. 

We can deduce the following Lemma quite quickly from the above facts. 
\begin{lem} \label{max}There exists a constant $C_m$ depending only on $\mathcal{C}, \alpha$ such that $|\max\{\ell_1, 0\} - \max\{\ell_2, 0\}| \leq C_m$. 
\end{lem}
\begin{proof} There exists a constant $c_f$ depending only of $f$ such that $|h(z) - \hat{h}_f(z)| \leq c_f$ for all $z \in \overline{\Q}^*$. Here $h$ is the usual logarithmic Weil height. Now combine the equalities (\ref{1}), (\ref{2}) with (\ref{height}). 
\end{proof}

In the proof of the next Lemma we make use of \cite[Theorem 2.1]{dynbogomolovcurves}.
\begin{lem}\label{bogomolov}
There exists a positive constant $C_B$ depending only on $\mathcal{C}, \alpha$ such that  $\max\{\ell_1, \ell_2\} \geq -C_{B}$. 
\end{lem}
\begin{proof}
From \cite[Theorem 1.5]{dynbogomolovcurves} and (\ref{1}), (\ref{2}) follows that unless $\mathcal{C}$ is  pre-periodic under $(f,f)$, there exists a constant $c >0$ depending only on $\mathcal{C}$ such that $\hat{h}_f(\beta_1) + \hat{h}_f(\beta_2) \geq c$. From (\ref{1}), (\ref{2}) follows immediately that  $\max\{m_1 - n_1, m_2-n_2\} \geq -C_{B}$ for some positive constant $C_B$. 
\end{proof}
As a corollary of Lemma \ref{min}, \ref{max}, \ref{bogomolov} we obtain the following. 
\begin{corollary} \label{diff}
There exists a constant $C_{\text{diff}} > 0$ depending only on $\alpha$ and $\mathcal{C}$ such that for $\ell_1,\ell_2$ as above holds that 
\begin{align*} 
|\ell_1 - \ell_2| \leq C_{\text{diff}}, ~ \min\{l_1,l_2\} \geq - C_{\text{diff}}.
\end{align*} 
\end{corollary}
In the following Lemma we make use of the uniformity of Proposition \ref{prop1}.
\begin{lem}\label{n1n2} For $n_1,n_2,m_1, m_2$ as in (\ref{rel1}), (\ref{rel2}) holds that 
\begin{align*} 
\max\{n_1,n_2\} \leq C_\text{inv}
\end{align*}
for a constant $C_{\text{inv}}$ depending only on $\alpha$ and $\mathcal{C}$. 
\end{lem}
\begin{proof} 
Let $M$ be an integer, larger than $C_{\text{diff}}$ and consider the finite union of curves 
\begin{align*}
S_{\mathcal{C}} = \{(f^{\circ M + k}, f^{\circ M})(\mathcal{C}), 0\leq k \leq C_{\text{diff}}\}\cup\{f^{\circ M}, f^{\circ M + k})(\mathcal{C}), 0\leq k\leq C_{\text{diff}}\}.
\end{align*}
By Proposition \ref{prop1} there exists a constant $N_{\mathcal{C}}$ such that if $(\gamma_1, \gamma_2) \in \tilde{\mathcal{C}}\cap \mathcal{S}_{\beta}^2$ for some $\beta \in \mathcal{O}_\alpha$, then $f^{N_{\mathcal{C}}}(\gamma_1) = f^{N_{\mathcal{C}}}(\gamma_2) = f^{N_{\mathcal{C}}}(\alpha)$. However, by Corollary \ref{diff} it holds that if $(\beta_1,\beta_2) \in \mathcal{C}\cap \mathcal{G}_\alpha^2$ then there exists $\tilde{\mathcal{C}} \in S_{\mathcal{C}}$ such that $(f^{k_1}(\beta_1), f^{k_2}(\beta_2)) \in \tilde{\mathcal{C}}\cap \mathcal{S}_\beta^2$ for some $\beta \in \mathcal{O}_\alpha$ and $0\leq k_1,k_2 \leq M+ C_{\text{diff}}$. This concludes the proof of the Lemma. 
\end{proof}
\subsection{Concluding the proof}
From Lemma \ref{n1n2} and Corollary \ref{diff} follows that 
\begin{align*}
\mathcal{C}\cap \mathcal{G}^2_\alpha \subset  \bigcup_{f_1 \leq C_{\text{diff}},  f_2 \leq C_{\text{diff}} }\bigcup_{|n_1|\leq C_{\text{inv}}, |n_2| \leq C_{\text{inv}}}\left((f^{-\circ n_1}, f^{-\circ n_2})(\mathcal{C})\cap S_{f_1,f_2}\right)
\end{align*}
 where 
\begin{align*}
S_{f_1,f_2}: f^{\circ f_1}(X) = f^{\circ f_2}(Y).
\end{align*}
(In fact it is enough to consider the curves for which $f_1 = 0$ or $f_2 = 0$). Now if the left hand side is infinite then the intersection of two curves in the union of curves $(f^{-\circ n_1},f^{-\circ n_2})(\mathcal{C}) \cap S_{f_1,f_2}$ contains a one dimensional component.  It follows that there exist $n,m$ such that $f^{\circ n}(\xi ) = f^{\circ m}(\eta)$ for the coordinate functions $\xi, \eta $ of $\mathcal{C}$ which contradicts our assumption on $\mathcal{C}$ not to be pre-periodic.
\section{Concluding remarks}\label{remarks}
In this section we discuss some issues surrounding effectivity and uniformity as well as future work. \\

It seems to be feasible to get an entirely effective version of Theorem \ref{thm}. The main issue seems to be to make Proposition \ref{prop1} effective. Almost all steps in the proof of Proposition \ref{prop1} can be made effective quite straightforwardly but especially in the case of Lemma \ref{upper2} more work is required as its proof uses something akin to a compactness argument. \\
In order to make Theorem \ref{equipotential} effective we can employ some sort of effective implicit function theorem such as in \cite[p.124]{langreal} and estimate the partial derivatives of the function $P(\Psi(z), \Psi(w))$ carefully which seems to be feasible to be done effectively.\\

A stumbling block on the way to effectivity for Theorem \ref{grandorbit} is that we use the results on the dynamical Bogomolov conjecture by Ghioca, Nguyen and Ye. However if the degree  $d$ is prime there is some hope that the methods in this paper might suffice. \\

For our final theorem we use  $o$-minimality (perhaps some weaker machinery would suffice) to find a uniform version of Theorem    \ref{equipotential}. 
\begin{thm}\label{uniform} There exists a constant $C_D$ depending only on $D$ such that any curve of total degree at most $D$ satisfies 
 \begin{align*}
 |\mathcal{C}\cap L_r^2|\leq C_D
 \end{align*}
for $\exp(-r)<  R/2$ unless $\mathcal{C}$ has the form in the conclusion of Theorem \ref{equipotential}
\end{thm}
 \begin{proof}
 The complex-analytic function $\Psi(z)$ is definable in the structure $\mathbb{R}_{an}$ on the open disc  of radius $ R/2$ and the statement follows from standard properties of $o$-minimal structures since the sets 
\begin{align*}
\{(z,w) \in \C^2; P(\Psi(z), \Psi(w)) = 0, |z|= |w| = \phi\}
\end{align*}
as $P$ varies over polynomials of degree at most $D$ and $\phi$ in the interval $(0,R/2)$ form a definable family. There is a uniform bound on the number of connected components of the members of a definable family in $\R_{an}$ \cite{ominimal} and we deduce the present Theorem from Theorem \ref{equipotential}.   
 \end{proof}
 Theorem \ref{uniform} implies that we have a similar uniformity in orbits as in Proposition \ref{prop1}. However we are missing the strong  Galois bounds or in fact any arithmetic Galois-information in order to extend Theorem \ref{grandorbit} to infinite places. This is an interesting direction for future research. \\
 
 We can prove variations of Theorem \ref{thm} in that we can mix our statement with the dynamical Bogomolov conjecture.  For example our methods seem to allow to prove structure theorems for $\mathcal{S}_\alpha \times \text{Preper}$ (perhaps under some conditions on $\alpha$) where $\text{Preper}$ is the set of pre-periodic points.  This will be explored in a future paper. \\

 We now show that an extension of Theorem \ref{grandorbit} to pre-periodic $\alpha$ is not possible. For this let $g$ be a polynomial of degree $2$ defined over the reals, whose Julia set is connected and let $f = g^{\circ 3}$. Then every periodic external ray of $f$ lands \cite[Theorem 18.1]{milnor}. Note that $\Psi(z^2) = g(\Psi(z))$. Let $\zeta_3$ be primitive third root of unity and let 
 \begin{align*}
 E= \{\Psi(r\zeta_3); r \in (1,\infty)\}.
 \end{align*}
 We set $\alpha$ to be the limit point of $E$ as $r \rightarrow 1$. Then $f^{\circ 2}(\alpha) = \alpha$ as $f^{\circ 2}(E) = E$ and for
 \begin{align*}
 \mathcal{C}: (X,g(X))
 \end{align*}
holds $|\mathcal{C}\cap \mathcal{G}_\alpha^2| = \infty$. To see this we can consider  $\alpha_n$ the limiting point of $E_n = \{\Psi(r\zeta); r \in (1,\infty)\}$ where $\zeta = \zeta_{8^n}\zeta_3$ for a primitive $8^n$-th root of 1, $\zeta_{8^n}$. Then $f^{\circ n}(E_n) = f^{\circ n}(E)$ and thus $f^{\circ n}(\alpha_n) = f^{\circ n}(\alpha)$. But it also hols that  $f^{\circ 2n +1}(g(E_n)) = E_n$ and thus $(\alpha_n ,g(\alpha_n)) \in \mathcal{G}_\alpha^2$ for all $n \geq 1$. Moreover the set of $\alpha_n$ just constructed is infinite \cite[Lemma 18.3]{milnor}.\\
But the variety $\mathcal{C}$ is not in the zero-locus of $f^{\circ n}(X)- f^{\circ m}(Y)$ for any choice of non-zero integers $n,m$ and thus we can not extend Theorem \ref{grandorbit} to pre-periodic points. 
\subsection{A conjecture for grand orbits}
In this subsection we want to generalise Conjecture \ref{conj} to all projective varieties, following a suggestion of Ghioca. It is modelled after the conjecture of Ghioca and Tucker on pre-periodic varieties \cite{dmmconj}. The novelty is that we need to define an analogue of a co-set for dynamical systems.\\
We assume that everything is defined over some algebraically closed field $\mathcal{K}$ of characteristic 0. 
We let $\mathcal{X}$ be a projective variety and let $f: \mathcal{X} \rightarrow \mathcal{X}$ be an endomorphism of degree at least 2. We first give a definition of admissible tuple associated to $f, \mathcal{X}$. 

\begin{definition} Let $\mathcal{V} \subset \mathcal{X}$ be a subvariety. A tuple $(\mathcal{Y}, h, H,\tilde{\mathcal{V}}, \varphi )$ is admissible for $(\mathcal{X}, f, \mathcal{V})$ if the following holds. Firstly $\mathcal{Y}$ is a projective variety and $h, H$ are endomorphism of $\mathcal{Y}$ of degree at least 2 that commute.  Moreover $\varphi$ is a rational map $\varphi: \mathcal{Y} \rightarrow \mathcal{X}$ with finite fibres that satisfies $\varphi \circ f^{\circ n} = h\circ \varphi $ for some $n \geq 0$, as well as $\varphi(\tilde{\mathcal{V}}) = \mathcal{V}$.
\end{definition}
Note that we can take $\mathcal{Y} = \mathcal{Y}$ and $H=h = f, \tilde{\mathcal{V}} = \mathcal{V}$ as well as $\varphi$ to be the identity to obtain an admissible tuple. Now we can define a notion of a variety to be dynamically weakly special with respect to $f$. 
\begin{definition} We say that a subvariety $\mathcal{V}\subset \mathcal{X}$ is dynamically weakly special with respect to $f$ if there exists an admissible tuple $(\mathcal{Y}, h, H,\tilde{\mathcal{V}}, \varphi )$ for $(\mathcal{X}, f, \mathcal{V})$ such that either $\tilde{\mathcal{V}}$ is pre-periodic by $H$ or  $\mathcal{Y} = \mathcal{Y}_1 \times \mathcal{Y}_2$ (with $\dim(\mathcal{Y}_2) > 0$) and $H(\tilde{\mathcal{V}}) \subset \mathcal{Y}_1\times \{\alpha\}$ for some $\alpha \in \mathcal{Y}_2(\mathcal{K})$. 
\end{definition}
As in the introduction we define the grand orbit of a point $\alpha \in \mathcal{X}(\mathcal{K})$ by an endomorphism $f$ to be 
\begin{align*}
\mathcal{G}_\alpha(f) = \{\beta \in \mathcal{X}(\mathcal{K}); f^{\circ n}(\alpha) = f^{\circ m}(\beta) \text{ for some } n,m \in \Z_{\geq 0}  \}.
\end{align*}

Now let $\mathcal{X}_1, \dots, \mathcal{X}_n$ be projective varieties and $f_1, \dots, f_n$ be endomorphisms of $\mathcal{X}_1, \dots, \mathcal{X}_n$ respectively such that $f = (f_1, \dots, f_n)$ has degree at least 2. 
\begin{conjecture}\label{bigconj} Let $\mathcal{X} = \mathcal{X}_1 \times \dots \times \mathcal{X}_n$ and $(\alpha_1, \dots, \alpha_n) \in \mathcal{X}(\mathcal{K})$. If for an irreducible subvariety $\mathcal{V} \subset \mathcal{X}$ holds that 
\begin{align*}
\mathcal{V} \cap \mathcal{G}_{\alpha_1}(f_1)\times \dots \times \mathcal{G}_{\alpha_n}(f_n)
\end{align*}
is infinite, then $\mathcal{V}$ is dynamically weakly special with respect to $f$. 
\end{conjecture} 
If all points $\alpha_1, \dots, \alpha_n$ are pre-periodic then Conjecture \ref{bigconj} follows from Conjecture 1.2 in \cite{dmmconj} as we can take (in the notation there) $\mathcal{X} = X, \mathcal{V} = Z$ as well as $(Y, \Phi_{|Y},\Psi, i )$ as an admissable tuple where $i$ is the inclusion of $Y$ in $X$. 
\subsection*{Acknowledgements} A special thanks goes to Myrto Mavraki for answering a lot of questions about the arithmetic of dynamical systems and for reading a previous draft of this paper.  Another thanks to Holly Krieger for answering my first basic questions about the state of the art of the arithmetic of dynamical systems. It is also a great pleasure to thank my colleagues at Basel (some of them not there anymore) Gabriel Dill, Richard Griffon, Philipp Habegger, Lars Kühne, Myrto Mawraki, Julia Schneider and Robert Wilms for many discussions surrounding this paper. A big thanks goes to Laura Demarco and to Dragos Ghioca  for valuable advice and encouragement as well as helpful comments. Another big thanks goes to Federico Pellarin for answering questions about Mahler functions. Finally I would like to heartily thank Philipp Habegger and the mathematical department of Basel (and its staff) for continuing support on all fronts. I would also like to thank the referee for his careful reading and many helpful suggestions which have greatly improved the exposition.   
\bibliography{Galoisbounds}
\bibliographystyle{abbrv}

\end{document}